\documentclass{amsart}

\usepackage{amssymb,verbatim}
\usepackage{graphicx}
\usepackage{psfrag,color}
\usepackage{pxfonts}

\newtheorem{theorem}{Theorem}[section]
\newtheorem{lemma}[theorem]{Lemma}
\newtheorem{proposition}[theorem]{Proposition}

\theoremstyle{definition}
\newtheorem{definition}[theorem]{Definition}
\newtheorem{example}[theorem]{Example}

\theoremstyle{remark}
\newtheorem{remark}[theorem]{Remark}

\numberwithin{equation}{section}

\newcommand{\abs}[1]{\left|#1\right|}
\newcommand{\diff}{\operatorname{d}}
\newcommand{\Diff}{\operatorname{D}}
\renewcommand{\H}{\mathbb{H}}
\newcommand{\id}{\operatorname{id}}
\newcommand{\Iso}{\operatorname{Iso}}
\newcommand{\N}{\mathbb{N}}
\newcommand{\Nil}{\operatorname{Nil}}
\newcommand{\norm}[1]{\left|\left|#1\right|\right|}
\newcommand{\PSL}{\operatorname{PSL}}
\newcommand{\R}{\mathbb{R}}
\renewcommand{\S}{\mathbb{S}}
\newcommand{\SL}{\operatorname{SL}}
\newcommand{\T}{\operatorname{T}}
\newcommand{\twist}{\operatorname{twist}}
\newcommand{\turn}{\operatorname{turn}}
\newcommand{\vol}{\operatorname{area}}
\newcommand{\volf}{\operatorname{vol}}
\newcommand{\Z}{\mathbb{Z}}

\renewcommand{\phi}{\varphi}
\renewcommand{\theta}{\vartheta}

\newcommand{\la}{\langle}
\newcommand{\ra}{\rangle}

\begin{document}

\title{}
\title[Surfaces with genus one and their sisters,\today]
{Surfaces with constant mean curvature $1/2$ and genus one in $\H^2\times\R$}
\date{\today}
\author[Plehnert]{Julia Plehnert}
\address{Georg-August-Universit\"at G\"ottingen, Research Group Discrete Differential Geometry, Faculty of Mathematics, Lotzestrasse 16-18, 37083 G\"ottingen, Germany}
\email{j.plehnert@math.uni-goettingen.de}

\subjclass[2010]{Primary 53A10}

\begin{abstract}

We construct new constant mean curvature surfaces in $\H^2\times\R$. They arise as sister surfaces of Plateau solutions. It is a family of mean curvature $1/2$ surfaces with $k$ ends, genus $1$ and $k$-fold dihedral symmetry, $k\geq3$. The surfaces are Alexandrov-embedded.
\end{abstract}

\maketitle

\section{Introduction}

For minimal surfaces there exist different approaches to define them, for example as solution of Plateau's problem. In 1970 Lawson proved a correspondence between isometric surfaces in spaceforms: For example, each $H$-surface (mean curvature $H\equiv1$) in $\mathbb{R}^3$ corresponds isometrically to a minimal surface in $\mathbb{S}^3$ such that their Gau\ss{} maps coincide and their tangent vectors are rotated by $\pi/2$ (with respect to a proper interpretation of the tangent spaces). Two such surfaces are called conjugate cousins~(\cite{lawson1970}, \cite{brauckmann1993}, \cite{karcher1989}). Daniel proved an equivalent correspondence between surfaces in homogeneous $3$-manifolds in~\cite{daniel2007}. A special case of his theorem yields the so-called sister surfaces: Each simply connected surface with constant mean curvature (cmc) $H$ in $\Sigma^2(\kappa) \times \mathbb{R}$ corresponds to an isometric minimal surface in the homogeneous $3$-manifold \mbox{$E(\kappa+4H^2,H)$}. Hence, instead of proving the existence of a cmc surface in a product space directly, one solves a Plateau problem in a non-trivial Riemannian fibration with base curvature $\kappa+4H^2$ and considers its sister. 

In the case of periodic cmc surfaces in product spaces the conjugate surface construction may be outlined as follows: Consider a geodesic polygon the auxiliary $3$-manifold $E(\kappa+4H^2,H)$, which consists of vertical and horizontal edges. Solve the Plateau problem for the given curve. In the case of unbounded domains we have to take a limit of the sequence of minimal surfaces. Afterwards solve the period problem(s), i.e. choose the parameters of the geodesic polygon such that the resulting cmc sister surface has the desired properties. With those parameters we constructed a minimal disk whose sister surface is a fundamental piece of the cmc surface. We use Schwarz reflection to establish a complete smooth surface in a product space. In the last step we prove geometric properties of the surface, such as (Alexandrov-)embeddedness, the behaviour of the ends, the genus, etc.

One challenge in the construction in homogeneous manifolds is, that the normals coincide up to their vertical projection in the sense of a Riemannian fibration only. But there are more problems that occur in the construction. In order to prove the existence of periodic surfaces via conjugate construction, the Jordan curve that bounds the Plateau solution of disk-type has to consists of geodesics. The Daniel correspondence implies that a geodesic in the boundary of a minimal surface corresponds to a curvature line in the boundary of the isometric cmc surface, therefore Schwarz reflection applies and continues the surface smoothly, see~\cite{GK2010} and~\cite{MT2011}. In the case of conjugate minimal surfaces in $\mathbb{R}^3$ the total curvature of the curvature line in one surface is equal to the total turn of the normal along the corresponding geodesic in the conjugate surface. But in the case of cmc surfaces the total curvature depends also on the length of the curve. In order to solve period problems in horizontal planes, one has to control the total turn of the normal along horizontal curvature lines. For a cmc surface with $H\equiv1/2$ in $\mathbb{H}^2\times\mathbb{R}$ this is possible by a horocycle foliation of $\mathbb{H}^2$. 

In this paper we prove the existence of new mc $1/2$ surfaces in $\H^2\times\R$, which are $k$-noids (with handle). This work is based on the PhD thesis of the author,~\cite{plehnert2012}. It is organised as follows, we start with an introduction to the geometry of homogeneous $3$-manifolds and prove a formula for the vertical distance of a horizontal lift of a closed curve in Riemannian fibrations.

In Section~\ref{S:Plateau} we cite some properties of Plateau solutions in $3$-manifolds and prove a maximum principle for cmc graphs. Together with a result of Gro\ss{}e-Brauckmann and Kusner, of which we sketch the proof, this implies a Rad\'{o} theorem for $E(\kappa,\tau)$. In the subsequent section we state some facts about certain sister surfaces and have a closer look on related boundary curves, the sister curves. We prove the absense of boundary branch points under specified conditions. Section~\ref{S:refsurfaces} outlines some properties of known minimal surfaces in $\Nil_3$, which we use as barriers in our construction.

In the last section we prove the existence of a family of mc $1/2$ surfaces in $\H^2\times\R$ with $k$ ends, genus $1$ and $k$-fold dihedral symmetry, $k\geq3$, which are Alexandrov embedded. We solve an improper Plateau problem and two period problems in the construction.

%
Within the last years the theory of minimal in constant mean curvature surfaces in homogeneous $3$-manifolds has been evolved actively by many mathematicians. Abresch and Rosenberg introduced a generalized quadratic differential for immersed surfaces in product spaces and show its holomorphicity. Furthermore they classify those surfaces with vanishing differential,~\cite{AR2005}. Fernandez and Mira constructed a hyperbolic Gau\ss{} map to study mean curvature one half surfaces in $\H^2\times\R$,~\cite{FM2007b}. Recently Cartier and Hauswirth published a work, where they study constant mean curvature $1/2$ surfaces in $\H^2 \times\R$ with vertical ends~\cite{CH2012}. In~\cite{MT2011} Manzano and Torralbo constructed constant mean curvature surfaces which arise via conjugate construction from compact minimal surfaces in the Berger spheres. The idea of this work comes from the paper~\cite{GK2010} of Kusner and Gro\ss{}e-Brauckmann, which is still in progress. They discuss the conjugate construction for minimal and cmc surfaces in project spaces $\Sigma^2(\kappa)\times\R$. The solution of Plateau's problem in mean convex subsets of homogeneous $3$-manifolds is treated. As an example of conjugate Plateau construction they sketch the existence of surfaces with constant mean curvature analogous to the minimal Jorge-Meeks-$k$-noids in $\R^3$, see Section~\ref{S:knoid}. We follow the same construction idea by a sequence of compact minimal sections for our example and give a different proof of the existence in Section~\ref{SS:Plateau}.

\section{Homogeneous $3$-manifolds}\label{S:riem}
We construct cmc surfaces in $\H^2\times \R$ which arise from minimal surfaces in $\Nil_3$. Both Riemannian manifolds are homogeneous $3$-manifolds with $4$-dimensional isometry group, see~\cite{thurston1997} for a complete classification of homogeneous $3$-manifolds. 

Simply connected homogeneous $3$-manifolds with an at least $4$-dimensional isometry group besides $\H^3$ can be represented as a Riemannian fibration over a two dimensional space form $\Sigma(\kappa)$. Their fibres are geodesics and there exists a Killing field $\xi$ tangent to the fibres, called the vertical vector field. The translations along the fibres are isometries, hence the Killing field $\xi$ generates a subgroup $G$ of $\Iso(E)$. The integral curves of $\xi$ define a principal bundle with connection $1$-form $\omega(X)=\la X,\xi\ra$. The curvature form is $\Omega\coloneqq  \Diff \omega=\diff\omega+1/2[\omega,\omega]=\diff\omega$. The following equation holds for two arbitrary vector fields $X,Y$
\begin{align*}
\Omega(X,Y)&=\diff\omega(X,Y)= X \omega(Y) - Y \omega(X) - \omega\left([X,Y]\right)\\
&=\la\nabla_X Y,\xi\ra+\la Y,\nabla_X\xi\ra-\la\nabla_Y X,\xi\ra-\la X,\nabla_Y\xi\ra -\left(\la\nabla_X Y,\xi\ra-\la\nabla_Y X,\xi\ra\right)\\
&=\la\nabla_X\xi,Y\ra-\la\nabla_Y\xi,X\ra\\
&\overset{(*)}{=}  2\la\nabla_X\xi,Y\ra,
  \end{align*}
  where $(*)$ follows from the fact that $\xi$ is a Killing field. Moreover, the fact that $\xi$ is Killing implies $\xi\Omega(X,Y)=0$, i.e. $\Omega$ is constant along the fibers. Moreover, we claim: For $X^h=X-X^v$ we have $\Omega(X^h,Y^h)=\Omega(X,Y)$.  Therefore $\Omega$ induces a $2$-form $\underline{\Omega}=(\pi^{-1})^*\Omega$ on the base manifold $\Sigma$. To see the claim we compute
  \begin{align*}
\Omega(X^h,Y^h)&= 2\la\nabla_{X^h}\xi,Y^h\ra\\
&=2\la\nabla_{(X-X^v)}\xi,Y-Y^v\ra\\
&=2(\la\nabla_{X}\xi,Y\ra\underbrace{-\la\nabla_X\xi,Y^v\ra}_{=\la\nabla_{Y^v}\xi,X\ra}-\la\nabla_{X^v}\xi,Y\ra+\la\nabla_{X^v}\xi,Y^v\ra)\\
&=2\la\nabla_X\xi,Y\ra,
  \end{align*}
  because $\nabla_U \xi=0$ for any vertical vector field $U$.
  
  Since $\Sigma$ is oriented, there exists a $\pi/2$-rotation $J$ on $\T_y\Sigma$, it induces a $\pi/2$-rotation $R$ on the horizontal space $(\T_p E)^h$.
\begin{definition}\label{d:bundle}
Let $\pi\colon E\to \Sigma$ be a Riemannian fibration with geodesic fibres. Its \emph{bundle curvature}~$\tau$ is a map $\tau\colon\Sigma\to\R$ given by
  \[
 \tau(y)\coloneqq-\frac12\Omega(X,RX)=\frac12\la[X,RX],\xi\ra,\]
 hence $[X,RX]^v = 2\tau(y) \xi$, where $X$ is an arbitrary horizontal unit vector field along $\pi^{-1}(y)$.
\end{definition}

We see that $\Omega$ measures the non-integrability of the horizontal distribution. Since $\Omega$ is constant along the fibres, the map $\tau$ is well-defined. The induced $2$-form $\underline{\Omega}$ factorizes the natural volume form $\volf_\Sigma$ of $\Sigma$ 
\[-\frac12\underline{\Omega}=\tau\volf_\Sigma,\] since we have $\volf_\Sigma(x,Jx)=1$ and $\underline{\Omega}(x,Jx)=\Omega(\tilde{x},R\tilde{x})$ for any unit vector $x$ on $\Sigma$.

If the Riemannian fibration is homogeneous, the bundle curvature $\tau$ is constant. It is common to write $E(\kappa,\tau)$ for those simply connected spaces. The isometry group of $E(\kappa,\tau)$ depends on the signs of $\kappa$ and $\tau$, and is equivalent to the isometry group of one of the following Riemannian manifolds:
  
\begin{center}\begin{tabular}{l | l l l}
  curv. & $\kappa<0$ & $\kappa=0$ & $\kappa>0$ \\ \hline
  $\tau=0$        & $\H^2\times\R$ & $\R^3$ & $\S^2\times\R$ \\
  $\tau\ne0$    & $\widetilde{\SL}_2(\R)$   & $\Nil_3(\R)$ & (Berger-)$\S^3$  
  \\ 
\end{tabular}
\end{center}
Another interpretation of the bundle curvature is the vertical distance of a horizontal lift of a closed curve:
\begin{lemma}[Vertical distances]\label{l:vertdistances}
Let $\gamma$ be a closed Jordan curve in the base manifold $\Sigma(\kappa)$ of a Riemannian fibration $E(\kappa,\tau)$ with geodesic fibres and constant bundle curvature $\tau$. With $\Delta$ defined by $\partial\Delta=\gamma$, we have
\[
d(\tilde{\gamma}(0),\tilde{\gamma}(l))=2\tau\vol(\Delta),
\]
where $\tilde{\gamma}$ is the horizontal lift with $\pi(\tilde{\gamma}(0))=\pi(\tilde{\gamma}(l))$, $\vol$ is the oriented volume and $d(p,q)$ denotes the signed vertical distance, which is positive if $\overline{pq}$ is in the fibre-direction $\xi$.
\end{lemma}
\begin{proof}
We consider an arclength parametrization $\gamma\colon[0,l]\to\Sigma(\kappa)$ and its horizontal lift $\tilde{\gamma}$. Then $\tilde{\gamma}(0)$ and $\tilde{\gamma}(l)$ are contained in one fibre, i.e. $\pi(\tilde{\gamma}(0))=\pi(\tilde{\gamma}(l))$. Hence, there exists a vertical arclength parametrized curve $c$ with $c(0)=\tilde{\gamma}(l)$ and $c(v)=\tilde{\gamma}(0)$. The union $c\cup\tilde{\gamma}\eqqcolon\tilde{\Gamma}$ is a closed curve in $E(\kappa,\tau)$ and $c'$ is parallel to $\xi$. If $c'=\pm\xi$, then
\[
\pm v
=\int\limits_0^v \la c',\xi\ra=\int\limits_0^v \la c',\xi\ra+\int\limits_0^l\la \tilde{\gamma}',\xi\ra,
\]
since $\tilde{\gamma}'$ is horizontal. By definition of the connection $1$-form $\omega$, the sum of the integrals is equal to $\int\limits_{\tilde{\Gamma}}\omega$.
We apply Stokes's Theorem to get
\[
\int\limits_{\tilde{\Gamma}}\omega=\int\limits_{\tilde{\Delta}}\diff\omega=\int\limits_{\tilde{\Delta}}\Omega=\int\limits_{\pi(\tilde{\Delta})}(\pi^{-1})^*\Omega,
\]
where $\tilde{\Delta}$ is any lift of $\Delta$, such that $\pi\colon\tilde{\Delta}\to\Delta$ is one-to-one and $\partial\tilde{\Delta}=\tilde{\Gamma}$. Since $(\pi^{-1})^*\Omega$ is a $2$-form on $\Sigma$, it factorizes the natural volume form $\volf_\Sigma$
\[\int\limits_{\pi(\tilde{\Delta})}(\pi^{-1})^*\Omega=-2\int\limits_\Delta\tau \volf_\Sigma=-2\tau\vol(\Delta).\]
We conclude that
\[
2\tau\vol(\Delta)=
\begin{cases}
-v, &  \text{if } c'=\xi, \\
 v, & \text{if } c'=-\xi,
\end{cases}
\]
and $d(\tilde{\gamma}(0),\tilde{\gamma}(l))=d(c(v),c(0))=2\tau\vol(\Delta)$.
\end{proof}

\section{Solution of the Plateau problem}\label{S:Plateau}
From the 18th century until 1930 it was an open question whether a rectifiable closed curve $\Gamma$ in $\R^3$ bounds an area minimizing disc. Douglas~\cite{douglas1931} and Rad\'{o}~\cite{rado1930} proved independently that there always exists an area minimizing disc spanned by $\Gamma$. In 1948 Morrey generalised the theorem to minimal discs in homogeneously regular Riemannian manifolds without boundary~(\cite{morrey1966}). By~\cite{osserman1970} and~\cite{gulliver1973}, the least area disc is a minimal immersion in the interior. Under additional assumptions we may exclude boundary branch points, see Section~\ref{SS:reflection} later on.

In the 80ties Meeks and Yau~(\cite{MY1982}) showed that in compact manifolds with mean convex boundary the Plateau solution $M$ is even embedded. A Riemannian manifold $N$ with boundary is \emph{mean convex} if the boundary $\partial N$ is piecewise smooth, each smooth subsurface of $\partial N$ has non-negative mean curvature with respect to the inward normal, and there exists a Riemannian manifold $N'$ such that $N$ is isometric to a submanifold of $N'$ and each smooth subsurface $S$ of $\partial N$ extends to a smooth embedded surface $S'$ in $N'$ such that $S'\cap N=S$ and two surfaces meet transversally at the non-smooth points of $\partial N$. We call each surface $S$ a \emph{barrier}.

One interesting question in the context of Plateau's problem is to ask, how many minimal surfaces of disc type are defined by a given closed Jordan curve? In general the answer is unknown. Rad\'{o} proved in~\cite{rado1930}: If the Jordan curve $\Gamma$ is graph over the boundary of a convex domain $\Delta\subset\R^2$, then $\Gamma$ bounds at most one minimal surface of disc type and the minimal surface is graph above $\Delta$. We consider Plateau's problem in mean convex domains in $E(\kappa,\tau)$. In a Riemannian fibration we have a natural notion of graphs, i.e. a section of the bundle $\pi\colon E(\kappa,\tau)\to\Sigma(\kappa)$. The following two propositions imply Rad\'{o} theorem for $E(\kappa,\tau)$.
In~\cite{GK2010} was proven if the boundary of a minimal surface projects to the boundary of a convex disc then it is a section:
\begin{proposition}
Let $\Delta\subset\Sigma(\kappa)$ be a convex domain in the base of the Riemannian fibration $E(\kappa,\tau)$. Suppose $\overline{M}\subset\overline{\pi^{-1}(\Delta)}$ is a compact minimal disc. If the boundary projection $\pi\colon\partial M\to\partial\Delta$ is injective, then $M$ is a section over $\Delta$.
\end{proposition}
For the sake of completeness we sketch the proof.
\begin{proof}
Let $\Omega\coloneqq \pi^{-1}(\Delta)$ and $\Phi_t\colon \Omega\to\Omega$ be the flow of the vertical Killing field $\xi$, where $t\in\R$. With $M_t\coloneqq \Phi_t(M)$ we have $M_0=M$. The surface $M$ is a section if $M_t\cap M_s=\emptyset$ for $t\ne s$. It sufficies to show $M\cap M_t=\emptyset$ for all $t$.

Assume the contrary, there exists $T\coloneqq \inf\{t>0\colon M\cap M_t=\emptyset\}>0$. This implies the existence of $p\in \overline{M}\cap\overline{M_T}$. By the maximum principle $\pi(p)\notin\Delta\setminus\partial\Delta$. But for $\pi(p)\in\partial\Delta$ follows that $p$ is an interior point for at least one of the surfaces, since the boundary projection $\pi\colon\partial M\to\partial\Delta$ is injective. This surface is then tangential one-sided to a vertical plane, and by the maximum principle it agrees with a subset of the vertical plane. This contradicts the injectivity of the boundary projection.

The case $T\coloneqq \sup\{t<0\colon M\cap M_t=\emptyset\}<0$ is analogous.
\end{proof}

The maximum principle implies uniqueness of constant mean curvature sections with the same boundary data:
\begin{proposition}\label{p:uniquesection}
Let $\pi\colon E\to\Sigma$ be a Riemannian fibration with geodesic fibres. Suppose $M$ is a section over $\Delta\subset\Sigma$ with mean curvature $H$ and prescribed boundary values such that $\pi\colon\partial M\to\partial\Delta$ is injective. Then $M$ is unique.
\end{proposition}
\begin{proof}
Assume that we have two sections $M$ and $\hat{M}$ over $\Delta$ with the same boundary values given by $u$ and $\hat{u}$. Moreover, assume $w\coloneqq u-\hat{u}\not\equiv 0$, meaning $\abs{w}$ has a maximum in an interior point $p\in\Delta\setminus\partial\Delta$, since $w\vert_{\partial\Delta}\equiv 0$. By exchanging $u$ and $\hat{u}$ we may assume $w\leq w(p)>0$.

The sections fulfil the non-parametric mean curvature equation. Therefore, their parametrizations $u$ and $\hat{u}$ are solutions of the following differential equation:
\[
Q(u)\coloneqq\partial_x\left(\frac{\lambda U}{W}\right)+\partial_y\left(\frac{\lambda V}{W}\right)-2\lambda^2 H=0,
\]
where $U=u_x+\lambda\tau y,\, V=u_y-\lambda\tau x,\, W=\sqrt{1+U^2+V^2}$ and $\lambda=\frac{4}{4+\kappa(x^2+y^2)}$.

This equation is non-linear. We set $a^i(p)\coloneqq\frac{\lambda p_i}{\sqrt{1+p_1^2+p_2^2}}$ and $R\coloneqq U\partial_x+V\partial_y$. Considering the difference we get:
\begin{align*}
0&=Q(u)-Q(\hat{u})=\sum\limits_i \partial_i a^i(R)-\partial_i a^i(\hat{R})\\
&=\sum\limits_i\partial_ia^i(tR+(1-t)\hat{R})\vert_{t=0}^{t=1}\\
&=\int\limits_0^1\frac{\diff}{\diff s}\left[\sum\limits_i\partial_ia^i(sR+(1-s)\hat{R})\right]_{s=t}\diff t\\
&=\sum\limits_{i,j}\int\limits_0^1\partial_i
	\left[
		\frac{\partial a^i}{\partial p_j}(tR+(1-t)\hat{R})(R_j-\hat{R}_j)
	\right]\diff t\\
&=\sum\limits_{i,j}\partial_i\left[\underbrace{\left(\int\limits_0^1\frac{\partial a^i}{\partial p_j}(tR+(1-t)\hat{R})\diff t\right)}_{\eqqcolon a^{ij}}(R_j-\hat{R}_j)\right]\\
&=\sum\limits_{i,j}\partial_i [a^{ij}\partial_j(u-\hat{u})].
\end{align*}
Since 
\[
\partial_ja^i(p)=\frac{\lambda\delta_{ij}}{\sqrt{1+p_1^2+p_2^2}}-\frac{\lambda p_i p_j}{\sqrt{1+p_1^2+p_2^2}^3}=\partial_ia^j(p),
\]
then $(a^{ij})_{ij}$ is symmetric, moreover $\abs{(\partial_i a^{ij})_j}$ is bounded. Therefore with $w= u-\hat{u}$, we have that
\[
L(w)\coloneqq \sum\limits_{i,j}\partial_i (a^{ij}\partial_jw)=Q(u)-Q(\hat{u})=0
\]
is a linear second order partial differential operator.

With $T(t,x,y)\coloneqq tR(x,y)+(1-t)\hat{R}(x,y)$ we get
$
\partial_j a^i(T)=\frac{\lambda\delta_{ij}}{\sqrt{1+\abs{T}^2}}- \frac{\lambda T_i T_j}{\sqrt{1+\abs{T}^2}^3}
$.
For any compact subset $K\subset\Delta$ the norm $\abs{T}$ has a maximum on $[0,1]\times K$. Hence, there exists $\sigma(K,u,\hat{u})>0$, such that we may estimate, by means of the Schwarz inequality,
\[
\sum\limits_{i,j}a^{ij}(T)\xi_i\xi_j=\frac{(1+\abs{T}^2)\abs{\xi}^2-\la T,\xi\ra^2}{\sqrt{1+\abs{T}^2}^3}\geq\frac{\abs{\xi}^2}{\sqrt{1+\abs{T}^2}^3}>\sigma\abs{\xi}^2.
\]
We conclude that $L$ is uniformly elliptic and therefore, the maximum principle for elliptic partial differential equations~(\cite{GT2001}) applies.

Hence, $w\equiv w(p)$ which contradicts $w\vert_{\partial\Delta}\equiv 0$.
\end{proof}

\begin{remark}
The uniqueness of a section is also true in a more general case: The projection of $\partial M$ has to be injective except for at most finitely many points of $\partial \Delta$. This means, we allow vertical segments in the boundary. The proof needs a more general maximum principle by Nitsche; for $\R^3$ see~\cite{nitsche1975}.
\end{remark}

\section{Sister surfaces}\label{C:sistersurfaces}
Lawson made a great contribution in the study of constant mean curvature surfaces in 1970 when he showed the isometric correspondence between minimal and cmc surfaces in different space forms, for example between minimal surfaces in $\mathbb{S}^3$ and $H$-surfaces in $\R^3$, see~\cite{lawson1970}. By means of the reflection principle in $\mathbb{S}^3$ it was then possible to construct new cmc surfaces in the Euclidean space. In recent years mathematicians have grown their interest within surfaces in other ambient manifolds. In 2007 Daniel published a generalized Lawson correspondence for homogeneous manifolds~(\cite{daniel2007}), we use one special case of his correspondence:

\begin{theorem}[{\cite[Theorem 5.2]{daniel2007}}] \label{t:correspondence} There exists an isometric correspondence between an mc $H$-surface $\tilde{M}$ in $\Sigma(\kappa)\times\R=E(\kappa,0)$ and a minimal surface $M$ in $E(\kappa+4H^2,H)$. Their shape operators are related by \begin{equation}\label{e:shapeoperators}\tilde{S}=JS+H\id,\end{equation} where $J$ denotes the $\pi/2$ rotation on the tangent bundle of a surface. Moreover, the normal and tangential projections of the vertical vector fields $\xi$ and $\tilde{\xi}$ are related by
\begin{equation}
\la\tilde{\xi},\tilde{\nu}\ra=\la\xi,\nu\ra,\qquad J\diff f^{-1}(T)=\diff \tilde{f}^{-1}(\tilde{T}),
\end{equation}
where $f$ and $\tilde{f}$ denote the parametrizations of $M$ and $\tilde{M}$ respectively, $\nu$ and $\tilde{\nu}$ their unit normals, and $T$, $\tilde{T}$ the projections of the vertical vector fields on $\T M$ and $\T\tilde{M}$.
\end{theorem}

We call the isometric surfaces $M$ and $\tilde{M}$ \emph{sister surfaces}, or \emph{sisters} in short.

{\it Examples}:
\begin{enumerate}
\item For $H=0$ the surfaces $M$ and $\tilde{M}$ are conjugate minimal surfaces in $\Sigma(\kappa)\times\R$.
\item For $H\in(0,1/2)$ and $\kappa=-1$ we have $4H^2-1<0$ and therefore corresponds to a minimal surface in $\widetilde{\PSL}_2(\R)$ and its mc $H$-sister surface in $\H^2\times\R$.
\item For $H=1/2$ and $\kappa=-1$ an mc $1/2$-surface in $\H^2\times\R$ corresponds to a minimal surface in $E(0,1/2)=\Nil_3(\R)$.
\item Furthermore, for $H>1/2$ an mc $H$-surface in $\H^2\times\R$ results from a minimal surface in the Berger spheres $E(4H^2-1,H)$, since $4H^2-1>0$.
\end{enumerate}

\subsection{Reflection principles}\label{SS:reflection}
We want to apply Schwarz reflection to construct complete periodic cmc surfaces in $E(\kappa,\tau)$. It is well-known that Schwarz reflection extends minimal surfaces in space forms (see~\cite{lawson1970}) with respect to the space form symmetries. In $E(\kappa,\tau)$ the isometry group has at least dimension $4$, and there are symmetries, for which Schwarz reflection applies. We reflect across a geodesic $c\subset E(\kappa,\tau)$ or a totally geodesic plane $V\subset E(\kappa,\tau)$, i.e. the geometric interpretation is to send a point $p$ to its opposite point on a geodesic through $p$ that meets $c$ or $V$ orthogonally.

If $c$ is a horizontal or vertical geodesic of $E(\kappa,\tau)$, then a geodesic reflection across $c$ is an isometry. Moreover, in the product spaces $E(\kappa,0)=\Sigma(\kappa)\times\R$ a geodesic reflection across vertical or horizontal planes is an isometry. For those isometries we formulate \emph{Schwarz reflection principles}:

\begin{quote}
Suppose that a minimal surface is smooth up to the boundary, and the boundary contains a curve which is a horizontal or vertical geodesic of $E(\kappa,\tau)$. Then the geodesic is an asymptotic direction and the reflection preserves the principal curvatures, therefore it extends the surface smoothly.

Moreover, a cmc surface in a product space $E(\kappa,0)$, which is smooth up to the boundary, extends smoothly if the boundary contains a curve in a vertical or horizontal plane, provided the surface conormal is perpendicular to the plane, since the curve a curvature direction. A totally geodesic plane is called \emph{mirror plane}, and a curve in which the surface meets a mirror plane orthogonally is called \emph{mirror curve}.
\end{quote}

We construct surfaces by solving Plateau problems. The solution is an area minimizing map from a disc to $E(\kappa,\tau)$, continuous up to the boundary. By~\cite{osserman1970} and~\cite{gulliver1973} the solution does not have branch points in the interior, hence it is an immersion.

Schwarz reflection extends the surface smoothly, but after reflection branch points may occur. Under certain assumptions, we may exclude boundary branch points, i.e. the surface extends as a smooth immersion:
\begin{proposition}\label{P:boundarybranch}
Let $M\subset E(\kappa,\tau)$ be a minimal surface of disc-type continuous up to the boundary $\Gamma=\partial M$. If $\Gamma$ is a Jordan curve that consists of both horizontal and vertical geodesics, such that for each edge of $\Gamma$ exists a vertical plane or a horizontal umbrella $S'$ as barrier for $\Gamma$, i.e. $S'\cap\overline{M}\subset\Gamma$, and moreover at each vertex $v$, the angle is of the form $\pi/n_v$, with $n_v\ge 2$, and there exists an union $\Gamma_v$ of $n_v$ copies of $\Gamma$, obtained by succesive $\pi$ rotations about the edges, such that there is a barrier $S'$ for $\Gamma_v$ in $v$. Then $M$ extends smoothly without branch points by Schwarz reflection across $\Gamma$.
\end{proposition}
The proof relies on the Hopf boundary lemma:
\begin{proof}
Let us consider an almost conformal harmonic parametrization $f$ of $M$ and a point $p\in\Gamma$ in the interior of an edge with barrier $S'$. Since $S'\cap\overline{M}\subset\Gamma$ the Hopf boundary lemma implies $\diff f_p\ne0$, hence $f$ is an immersion.

We assume vertex $v\in\Gamma$ is a branch point, and consider the union $\Gamma_v$ of $n_v$ copies of $\Gamma$. By assumption there is a barrier $S'$ for $\Gamma_v$ and we are in the first case. Hence, we conclude $M$ extends smoothly without branch points by Schwarz reflection across $\Gamma$.
\end{proof}


\subsection{Sister curves}\label{S:sistercurves}
We want to analyse the geometry of periodic surfaces. Therefore, we take a closer look at the boundary curves of a fundamental piece.

Let $c=f\circ\gamma$ be a curve parametrized by arc length in a hypersurface $M=f(\Omega)\subset\overline{M}$ with (surface) normal $\nu$. The \emph{normal curvature} $k$ and the \emph{normal torsion} $t$ along $c$ are defined by
\[
k\coloneqq \nu\cdot\overline{\nabla}_{c'}c'=-\overline{\nabla}_{c'}\nu\cdot c'=\la S\gamma',\gamma'\ra ,\qquad 
t\coloneqq-\overline{\nabla}_{c'}\nu\cdot Jc'=\la S\gamma',J\gamma'\ra.
\]

Let $\tilde{M}\subset\Sigma(\kappa)\times\R$ denote an mc $H$-surface and $M\subset E(\kappa+4H^2,H)$ its minimal sister. Furthermore, let $\gamma$ be a curve in $\Omega$. We call $\tilde{c}\coloneqq \tilde{f}(\gamma)$ and $c\coloneqq f(\gamma)$ \emph{sister curves}.

\begin{lemma}\label{l:curvtor}
For a pair of sister curves the normal curvature and torsion are related as follows:
$$\tilde{k}=-t+H\quad\text{and}\quad\tilde{t}=k.$$
\end{lemma}
\begin{proof}
We apply Equation~\eqref{e:shapeoperators} to the definitions:
\begin{align*}
\tilde{k}&=\la \tilde{S}\gamma',\gamma'\ra=\la (JS+H\id)\gamma',\gamma'\ra=-t+H,\\
\tilde{t}&=\la \tilde{S}\gamma',J\gamma'\ra=\la (JS+H\id)\gamma',J\gamma'\ra=k.\qedhere
\end{align*}
\end{proof}
From this follows a relation between mirror curves and their sister curves, see~\cite{GK2010} and~\cite{MT2011}:
\begin{enumerate}
\item A curve $\tilde{c}\subset\tilde{M}\subset\Sigma(\kappa)\times\R$ is a mirror curve in a vertical plane if and only if its sister curve $c\subset M\subset E(\kappa+4H^2,H)$ is a horizontal geodesic.
\item Similarly, $\tilde{c}$ is a horizontal mirror curve if and only if $c$ is a vertical geodesic.
\end{enumerate}

In our construction we consider the \emph{fundamental piece} of a periodic mc $H$-surface. The complete surface is then generated by reflections. The fundamental piece is simply connected and bounded by mirror curves $\tilde{c}_i$ in vertical and horizontal planes. Given a fundamental piece bounded by $n$ arc length parametrized mirror curves $\tilde{c}_i$, it defines the following geometric quantities:
\begin{itemize}
\item The \emph{length} $\tilde{l}_i$ of the mirror curve $\tilde{c}_i$, also denoted by $l(\tilde{c}_i)$ or $\abs{\tilde{c}_i}$.
\item The \emph{vertex angle} $\tilde{\phi}_i$ of two edges $\tilde{c}_i$ and $\tilde{c}_{i+1}$, which satisfies $$\cos\tilde{\phi}_i=-\tilde{c_i}'(\tilde{l}_i)\cdot\tilde{c}_{i+1}'(0).$$
\end{itemize}
The minimal sister surface is bounded by horizontal and vertical geodesics. Since the surfaces are isometric, we have
\begin{equation}\label{e:equality}
\tilde{l}_i=l_i,\qquad\tilde{\phi}_i=\phi_i.
\end{equation}
\begin{itemize}
\item The total \emph{turn} angle of the normal $\tilde{\nu}$ $$\turn_i=\turn_{\tilde{c}_i}(\tilde{\nu})\coloneqq\int_{\tilde{c}_i}\tilde{k},$$ which measures the total turn of the normal relative to a parallel field.
\end{itemize}

Accordingly, we want to measure the rotational angle of the normal $\nu$ along a curve $c$ in $E(\kappa+4H^2,H)$. We detect the \emph{twist} of the normal with respect to an appropiate vector field $X$
\[\twist_c(\nu,X)\coloneqq\int\limits_c \la\nabla_{c'}\nu,c'\times\nu\ra-\la\nabla_{c'}X,c'\times X\ra.\] 
\begin{definition}\label{d:twist}
\begin{enumerate}
\item\label{d:verttw} Let $c\subset M$ be a vertical geodesic in $E(\kappa+4H^2,H)$, then the twist is defined by the total rotation speed of $\nu$ with respect to a basic vector field $\tilde{e}$, i.e. the horizontal lift of any vector field $e$ on $\Sigma(\kappa+4H^2)$:
\[\twist_v\coloneqq\twist_c(\nu,\tilde{e})=\int\limits_c \la\nabla_{c'}\nu,c'\times\nu\ra-\la\nabla_{c'}\tilde{e},c'\times \tilde{e}\ra.\]
\item Let $c\subset M$ be a horizontal geodesic in $E(\kappa+4H^2,H)$, then the twist is defined by the total rotation speed of $\nu$ with respect to the vertical vector field $\xi$:
\[\twist_h\coloneqq\twist_c(\nu,\xi)=\int\limits_c \la\nabla_{c'}\nu,c'\times\nu\ra-\la\nabla_{c'}\xi,c'\times \xi\ra.\]
\end{enumerate}
\end{definition}
With this definition $\twist_v$ measures the angle $\measuredangle(\diff\pi_c(\nu(0)),\diff\pi_{c}(\nu(l))$  in the projection. We drop the index when it is clear whether the geodesic is vertical or horizontal.

\begin{lemma}\label{l:torsionangle}
\begin{enumerate}
\item Let $c\subset M$ be a vertical geodesic in $E(\kappa+4H^2,H)$, whose sister is a horizontal mirror curve $\tilde{c}\subset\tilde{M}$ in $E(\kappa,0)$. Then $$\twist_v=\int\limits_c t+H l(c)\quad\mbox{and}\quad\tilde{k}=2H-\twist_v'.$$
\item Let $c\subset M$ be a horizontal geodesic in $E(\kappa+4H^2,H)$, whose sister is a vertical mirror curve $\tilde{c}\subset\tilde{M}$ in $E(\kappa,0)$. Then $$\twist_h=-\turn.$$
\end{enumerate}
\end{lemma}
\begin{proof}
\begin{enumerate}
\item Without loss of generality $c'=\xi$. Let $(c',J c',\nu)$ be positively oriented. Let $J$ and $R$ denote $\pi/2$-rotations in the tangent bundle $\T M$ and the horizontal plane $\diff \pi^{-1}(\T \Sigma(\kappa+4H^2))$, respectively. Then $Jc'$ and $\nu$ are horizontal with $-R\nu=Jc'$.

We denote by $E$ an unit basic vector field. We have
\begin{align*}
\twist_v &=\int\limits_c\left(\la\nabla_{c'}\nu,c'\times\nu\ra-\la\nabla_{c'}E,c'\times E\ra\right)\\
&=\int\limits_c\left(\la\nabla_{c'}\nu,R\nu\ra-\la\nabla_{c'}E,R E\ra\right)\\
&=\int\limits_c\left(-\la\nabla_{c'}\nu,J c'\ra-\la\nabla_{\xi}E,R E\ra\right)\\
&=\int\limits_c\left(t+H\la RE,R E\ra\right)\\
&=\int\limits_c t+Hl(c).
\end{align*}
Lemma~\ref{l:curvtor} implies that $\tilde{k}=2H-\twist_v'$.
\item The rotational angle of the tangent plane $\T _{c(t)}M$ is measured by the rotational speed with respect to $\xi$:
\[
\twist_h'=\la\nabla_{c'}\nu,\underbrace{c'\times\nu}_{=-Jc'}\ra-\la\underbrace{\nabla_{c'}\xi}_{=-HRc'},\underbrace{c'\times\xi}_{=-Rc'}\ra=t-H=-\tilde{k}.
\] 
By integrating along $c$ we get \[\twist_h=-\turn.\]
\end{enumerate}
\end{proof}
\begin{example}
We compute the torsion and the twist of a vertical geodesic $c$ in a vertical plane in $E(\kappa,\tau)$. The geodesic $c$ is a fibre of the Riemannian fibration. Without loss of generality $c'=\xi$, then we get
\begin{align*}
t&=-\nabla_{c'}\nu\cdot Jc'\\
&=\tau R\nu\cdot Jc'=-\tau.
\end{align*}
Namely, the torsion of a vertical geodesic is the negative of the bundle curvature.

Moreover, for the twist we get, as expected
\[
\twist=\int\limits_c t+\tau l(c)=0.
\]
Hence, with respect to parallel fields, the normal does not rotate. Equivalently the normal is constant in the projection.
\end{example}

We shall apply Lemma~\ref{l:torsionangle} to obtain detailed information about vertical geodesics in $\Nil$ and their horizontal sister curves in $\H^2\times\R$. This strategy is due to Laurent Mazet, for whom the author is very grateful.

For a curve $\tilde{c}\subset\H^2$, consider the unique horocycle foliation $\mathcal{F}_{\tilde{c}}$ given by the horocycle that is tangent to $\tilde{c}$ in $\tilde{c}(0)$ and has curvature $1$ with respect to the normal $n$ of $\tilde{c}$. Let $\theta$ be the angle defined by $\tilde{c}^{\prime}=\cos\theta e_1+\sin\theta e_2$, where the orthonormal frame $(e_1, e_2)$ is given by the tangent and minus the normal of the horocycles. By the definition of the considered foliation we have $\theta(0)=0$ and $n=\sin\theta e_1-\cos\theta e_2.$
\begin{figure}[h]
\begin{center}
\psfrag{0}{$y=0$}
\psfrag{1}{$e_1$}
\psfrag{2}{$e_2$}
\psfrag{n}{$n$}
\psfrag{c}{$\tilde{c}^{\prime}$}
\psfrag{t}{$\theta$}
\includegraphics[width=7cm]{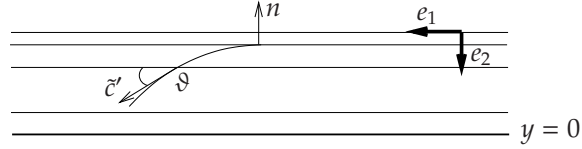}
\end{center}
\caption{Horocycle foliation $\mathcal{F}_{\tilde{c}}$ given by $\tilde{c}$.}\label{f:foliation}
\end{figure}
\begin{lemma}\label{l:curv}
The curvature of $\tilde{c}$ is given by $$\tilde{k}=\cos\theta-\theta'.$$
\end{lemma}
\begin{proof}
A simple computation gives:
\begin{align*}
\nabla_{\tilde{c}^{\prime}}\tilde{c}^{\prime}=&\nabla_{\cos\theta e_1+\sin\theta e_2}(\cos\theta e_1+\sin\theta e_2)\\
=&-\theta'\sin\theta e_1+\cos\theta\nabla_{(\cos\theta e_1+\sin\theta e_2)}e_1+\theta'\cos\theta e_2+\sin\theta\nabla_{(\cos\theta e_1+\sin\theta e_2)}e_2\\
=&-\theta'\sin\theta e_1+\cos^2\theta\underbrace{\nabla_{e_1}e_1}_{-e_2}+\cos\theta\sin\theta \underbrace{\nabla_{e_2}e_1}_{=0}\\
&+\theta'\cos\theta e_2+\sin\theta\cos\theta\underbrace{\nabla_{e_1}e_2}_{e_1}+\sin^2\theta \underbrace{\nabla_{e_2}e_2}_{0}\\
=&-\theta'\sin\theta e_1-\cos^2\theta e_2+\theta'\cos\theta e_2+\sin\theta\cos\theta e_1\\
=&(-\theta'+\cos\theta)\sin\theta e_1-(-\theta'+\cos\theta)(\cos\theta e_2)\\
=&(\cos\theta-\theta')n.
\end{align*}\qedhere
\end{proof}
We want to control the curve $\tilde{c}$ by means of its sister $c$ in $M\subset E(0,1/2)=\Nil$. Let $\alpha(t)=\twist_v$ measure the twist in $c(t)$ with respect to a basic vector field chosen such that $\alpha(0)=0$, i.e. $\tilde{e}(c(0))=\tilde{\nu}(c(0))$.
\begin{proposition}\label{p:anglecompare}
Let $\tilde{c}\subset\H^2\times\R$ be the horizontal sister curve of a vertical geodesic $c$ in $E(0,1/2)=\Nil$. The angle $\theta$ given by the horocycle foliation $\mathcal{F}_{\tilde{c}}$ and the rotational speed $\alpha'$ of the minimal surface normal along $c$ are related as follows: \[\theta'=\alpha'+\cos\theta-1,\quad\text{ and }\quad \theta\leq\alpha,\]
where $\alpha(t)$ measures the angle between $\nu$ and any parallel field chosen such that $\alpha(0)=0$.
\end{proposition}
\begin{proof}
We have seen in Lemma~\ref{l:torsionangle} that $\tilde{k}=1-\alpha'$. Together with Lemma~\ref{l:curv} we get $$\theta'=\alpha'+\cos\theta-1,$$ namely $\theta'\leq\alpha'$. In particular, along the curves $c$ and $\tilde{c}$ we have $$\int_{\tilde{c}}\theta^{\prime}\leq\int_c\alpha^{\prime}\Rightarrow \theta(t)\leq\alpha(t).$$\qedhere
\end{proof}

\section{Reference surfaces}\label{S:refsurfaces}
\subsection{Horizontal helicoids in $\Nil_3$}\label{S:conormalhoriheli}
In the construction of the $k$-noid with genus $1$ from Section~\ref{S:genus1}, we use as a barrier the horizontal helicoid $H_\alpha(u,v)$ in $\Nil$, $\left(\R^3,\diff x_1^2+\diff x_2^2+(\diff x_3-x_1\diff x_2)^2\right)$ by Daniel and Hauswirth~\cite[Section 7]{DH2009}. For $\alpha>0$ the coordinates of the helicoid $H_\alpha$ are given in terms of the solution $\psi$ of the ordinary differential equation $\psi'^2=\alpha^2+\cos^2\psi,\psi(0)=0$:
\begin{align*}
x_1&=\frac{\sinh(\alpha v)}{\alpha(\psi'(u)-\alpha)}\cos\psi(u)\\
x_2&=-G(u)\\
x_3&=\frac{-\sinh(\alpha v)}{\alpha(\psi'(u)-\alpha)}\sin\psi(u),
\end{align*}
where $G$ is defined by $G'(u)=1/(\psi'(u)-\alpha), G(0)=0$. In~\cite{DH2009} was shown that the function $\psi$ is a decreasing odd bijection. There exists a unique $U\coloneqq U(\alpha)>0$ with $\psi_\alpha(U)=-\pi/2$, $\psi_\alpha(-U)=\pi/2$. To visualise the surface, we look at three curves in the helicoid:
\begin{align*}
H_\alpha(-U,v)&=\left(0,G(U),\frac{\sinh(\alpha v)}{\alpha(\alpha-\psi'(-U))}\right)\\
H_\alpha(0,v)&=\left(\frac{\sinh(\alpha v)}{\alpha(\psi'(0)-\alpha)},0,0\right)\\
H_\alpha(U,v)&=\left(0,-G(U),\frac{\sinh(\alpha v)}{\alpha(\psi'-\alpha)}\right).
\end{align*}
The rulings $H_\alpha (\pm U,v)$ are vertical and
define the width $a\coloneqq G(-U)-G(U)$ of the helicoid. The width is well-defined for the whole helicoid, since $\psi$ is periodic: $\psi(u+2U)=\psi(u)-\pi$.
\begin{figure}[h]
\begin{center}
\psfrag{1}{$x_1$}
\psfrag{2}{$x_2$}
\psfrag{3}{$x_3$}
\psfrag{+}{$-G(U)$}
\psfrag{-}{$-G(-U)$}
\includegraphics[width=5.6cm]{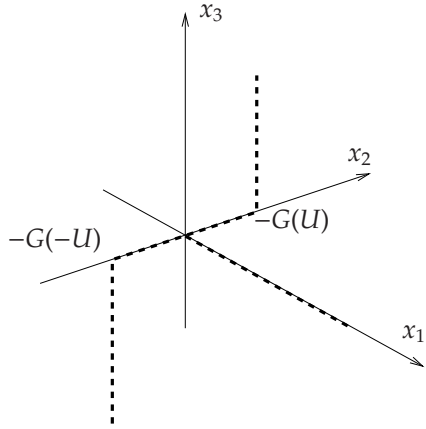}\end{center}
\caption{Sketch of a fundamental piece of the horizontal helicoid from Daniel and Hauswirth in $\Nil$, $v\leq0$.}\label{f:horiheli}
\end{figure}

To ensure that we can consider a helicoid $H_\alpha$ for a given width $a$, we need the following lemma:
\begin{lemma}\label{l:a0alphainf}
For $a>0$, there exists $\alpha>0$ such that $$-2G(U(\alpha))=a,$$ where $U(\alpha)$ is defined by $\psi_\alpha(U)=-\pi/2$. Furthermore, for $a\to 0$ we have $\alpha\to\infty$.
\end{lemma}
\begin{proof}
The idea of the proof is to show that $G$ is a bijection on $\R$ first. The second step is to show that $U$ is a continuous map to $\R_+$.

Step 1: The function $G$ is odd, so $G(-U)-G(U)=-2G(U)$. For $a>0$ we show that there exists exactly one $U>0$ such that $-2G(U)=a$: From $G'=1/(\psi'-\alpha)<0$ we know that $G$ is a decreasing function on $\R$. If we assume that $G$ is bounded, i.e. $G(u)\to g\in\R$ for $u\to\infty$, then $G'(u)\to 0$ for $u\to\infty$. But this implies $\psi'(u)-\alpha\to-\infty$ for $u\to\infty$, which is a contradiction because $\psi'^2=\alpha^2+\cos^2\psi\leq\alpha^2+1$ is bounded. Therefore, $G$ is a decreasing bijection on $\R$.

Step 2: We show for $U>0$ the existence of $\alpha>0$ such that the solution $\psi_\alpha$ of \[\psi'^2_{\alpha}=\alpha^2+\cos^2\psi_\alpha,\quad\psi_\alpha(0)=0\] satisfies $\psi_\alpha(U)=-\pi/2$.

By applying seperation of variables to the ODE $\psi'_\alpha=-\sqrt{\alpha^2+\cos^2\psi},\,\psi(0)=0$, the solution is then given by the inverse of the elliptic integral of the first kind
\[\int\limits_0^\psi -\frac{1}{\sqrt{\alpha^2+\cos^2\theta}}\diff\theta.\] 
We are interested in $U(\alpha)$ given by $\psi_\alpha(U)=-\pi/2$:
\begin{align*}
U(\alpha)&=\int\limits_0^{-\pi/2} -\frac{1}{\sqrt{\alpha^2+\cos^2\theta}}\diff\theta\\
&=\int\limits_0^{\frac\pi 2} \frac{1}{\sqrt{\alpha^2+\cos^2\theta}}\diff\theta\\
&=\frac{1}{\sqrt{\alpha^2+1}}\int\limits_0^{\frac\pi 2} \frac{1}{\sqrt{1-\frac{1}{\alpha^2+1}\sin^2\theta}}\diff\theta\\
&=\frac{K\left(1/\sqrt{\alpha^2+1}\right)}{\sqrt{\alpha^2+1}},
\end{align*}
where $K(k)=\int\limits_0^{\pi/2}\frac{\diff\theta}{\sqrt{1-k^2\sin^2\theta}}=\int\limits_0^1\frac{\diff t}{(1-t^2)(1-k^2t^2)}$ denotes the complete elliptic integral of the first kind, defined for $k\in[0,1)$, with the special values $K(0)=\pi/2$ and $\lim_{k\to 1}K(k)=\infty$. For $\alpha\to0$ we have $K\left(1/\sqrt{\alpha^2+1}\right)\to\infty$ and for $\alpha\to\infty$ we have $K\left(1/\sqrt{\alpha^2+1}\right)\to\pi/2$.
Therefore, $U$ is continuous because $K$ is. Moreover, for $\alpha\to0$ we have $U(\alpha)\to\infty$ and for $\alpha\to\infty$ we have $U(\alpha)\to0$.

Together with the Step 1, this concludes the first part of the lemma. For the second part, notice that $a\to0$ implies $U\to0$, since $G$ is a decreasing function with $G(0)=0$. Furthermore, $U(\alpha)\to0$ implies $\alpha\to\infty$, since $K$ is bounded.
\end{proof}

We want to express the height $b$ of the conormal $\eta$ of the helicoid $H_\alpha$ along the vertical rulings depending on the width $a$ and the angle $\phi$ in the horizontal plane $\operatorname{span}\{\partial x_1,\partial x_2+x_1\partial x_3\}$ given by 
\[\cos\phi=\la\partial x_1,\eta\ra.\] 
The conormal along the horizontal ruling \[H_\alpha(-U,v)=\left(0,G(U),\frac{\sinh(\alpha v)}{\alpha(\psi'(-U)-\alpha)}\right)\] is given by 
\[\frac{\partial H_\alpha}{\partial u}(-U,v)=\left(\frac{-\sinh(\alpha v)\psi'(-U)}{\alpha(\psi'(-U)-\alpha)},-G'(-U),0\right),\quad 
\nu=\frac{\partial_u H_\alpha}{\norm{\partial_u H_\alpha}}.\] The conormal along $H_\alpha(-U,v)$ is horizontal, since $x_1=0$. We may express $\phi$ in terms of $(\alpha,U=U(\alpha))$. By~\cite{DH2009} we have \[\psi'(-U)=G'(-U)\cos^2\psi(-U)-\alpha=-\alpha\quad \text{and}\quad G'(-U)=\frac{1}{\psi'(-U)-\alpha}=\frac{-1}{2\alpha}.\] Therefore \[\eta=\frac{2\alpha^2}{\sqrt{\alpha^2(\sinh^2(\alpha v)+1)}}\partial_u H_\alpha(-U,v)=\frac{2\alpha}{\cosh(\alpha v)}\partial_u H_\alpha(-U,v)\] and 
\[\cos\phi=\frac{2\alpha^2\sinh(\alpha v)}{-2\alpha^2\cosh(\alpha v)}=\tanh(-\alpha v),\,v\leq0\Leftrightarrow v=\frac{-1}{2\alpha}\ln\left(\frac{1+\cos\phi}{1-\cos\phi}\right).\] 
Using this we get the height $b$ of the conormal $\eta$, in terms of $\phi$ and $(\alpha,U(\alpha))$ for fixed $\alpha>0$, as follows:
\[b=\frac{\sinh\left(\frac{-1}{2}\ln\left(\frac{1+\cos\phi}{1-\cos\phi}\right)\right)}{2\alpha^2}\leq0.\]
One readily sees that $\phi\to0$ implies $b\to-\infty$, and $b\to0$ when $\phi\to\pi/2$.

\begin{remark}
For each $\alpha>0$ the fundamental piece of the helicoid $H_\alpha$ is a section of the Riemannian fibration $\pi\colon\Nil_3\to\R^2$ defined on 
$\R\times\left(-G(-U),-G(U)\right)\subset\R^2$.
\end{remark}

\subsection{Constant mean curvature $k$-noids in $\Sigma(\kappa)\times\R$}\label{S:knoid}

In~\cite[Section 5.]{GK2010} Gro\ss{}e-Brauckmann and Kusner sketched the construction of an one-parameter family of surfaces with constant mean curvature $H\geq0$ in $\Sigma(\kappa)\times\R$ with $\kappa\leq0$, which have $k$ ends, dihedral symmetry and genus zero. 

Their idea was to consider a sequence of compact Plateau solutions $M_{(r,s)}$, which represent sections in $E(\kappa+4H^2,H)$. Each minimal disc $M_{(r,s)}$ is bounded by horizontal and vertical geodesics, see Figure~\ref{f:knoid}. Let $\Gamma_{(r,s)}$ denote the boundary. The minimal surface $M_{(r,s)}$ is a section of the trivial line bundle \[\pi\colon\Omega_r\subset E(\kappa+4H^2,H)\to\Delta_r,\] where $\Omega_r\coloneqq\pi^{-1}(\Delta_r)$ is a mean convex domain, which is defined as the preimage of a triangle $\Delta_r\subset\Sigma(\kappa+4H^2)$. The triangle $\Delta_r$ is given by a hinge of lengths $a$ and $r$, enclosing an angle $\pi/k$. The parameter $a$ determines the length of the horizontal edge in the boundary of $M$, it corresponds to the necksize in the cmc sister.

\begin{figure}[h]\begin{center}
\psfrag{a}{$a$}
\psfrag{r}{$r$}
\psfrag{pr}{$\pi$}
\psfrag{s}{$s$}
\psfrag{phi}{$\pi/k$}
\includegraphics[width=4.5cm]{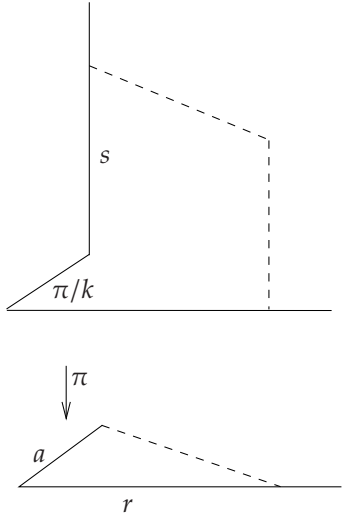}
\caption{The boundary of the minimal disc in $E(\kappa+4H^2,H)$.}\label{f:knoid}\end{center}\end{figure}

In order to show that the sequence of compact minimal surfaces $M_{(r,s)}$ has a minimal surface $M=M(a,k)$ with infinite boundary $\Gamma$ as a limit, such that $M$ is a section projecting to $\Delta\coloneqq\lim_{r\to\infty}\Delta_r$, and $M$ extends without branch points by Schwarz reflection about the edges of $\Gamma$, one has to show that there exist barriers. The proof is equivalent to the one of Theorem~\ref{t:plateauinfinite} below. The complete cmc surface is obtained by considering the sister and using Schwarz reflection.

\section{Constant mean curvature $k$-noids with genus $1$}\label{S:genus1}

We construct surfaces with mc $1/2$ in $\H^2\times\R$ with $k$ ends and genus $1$. Each surface has $k$ vertical symmetry planes and one horizontal symmetry plane, where $k\geq 3$. The idea is to solve a Plateau problem of disc type in $\Nil_3(\R)$, where the disc is bounded by geodesics. Its sister disc in $\H^2\times\R$ generates an mc $1/2$ surface by reflections about horizontal and vertical planes. The problem is to define the geodesic contour such that the sister has the desired properties.

\subsection{Boundary construction}

In $\H^2\times\R$ the desired boundary is connected and consists of four mirror curves in three symmetry planes; the two vertical symmetry planes form an angle $\pi/k$, see Figure~\ref{f:genus1contour}.
\begin{figure}[h]
\begin{center}
\psfrag{c1}{$c_1$}
\psfrag{c2}{$c_2$}
\psfrag{a}{$a$}
\psfrag{n}{$n$}
\psfrag{pr}{$\pi$}
\psfrag{phi}{$\phi$}
\psfrag{pi}{$\pi/k$}
\psfrag{ct1}{$\tilde{c}_1$}
\psfrag{ct2}{$\tilde{c}_2$}
\psfrag{p}{$\hat{p}_1$}
\includegraphics[width=0.5\textwidth]{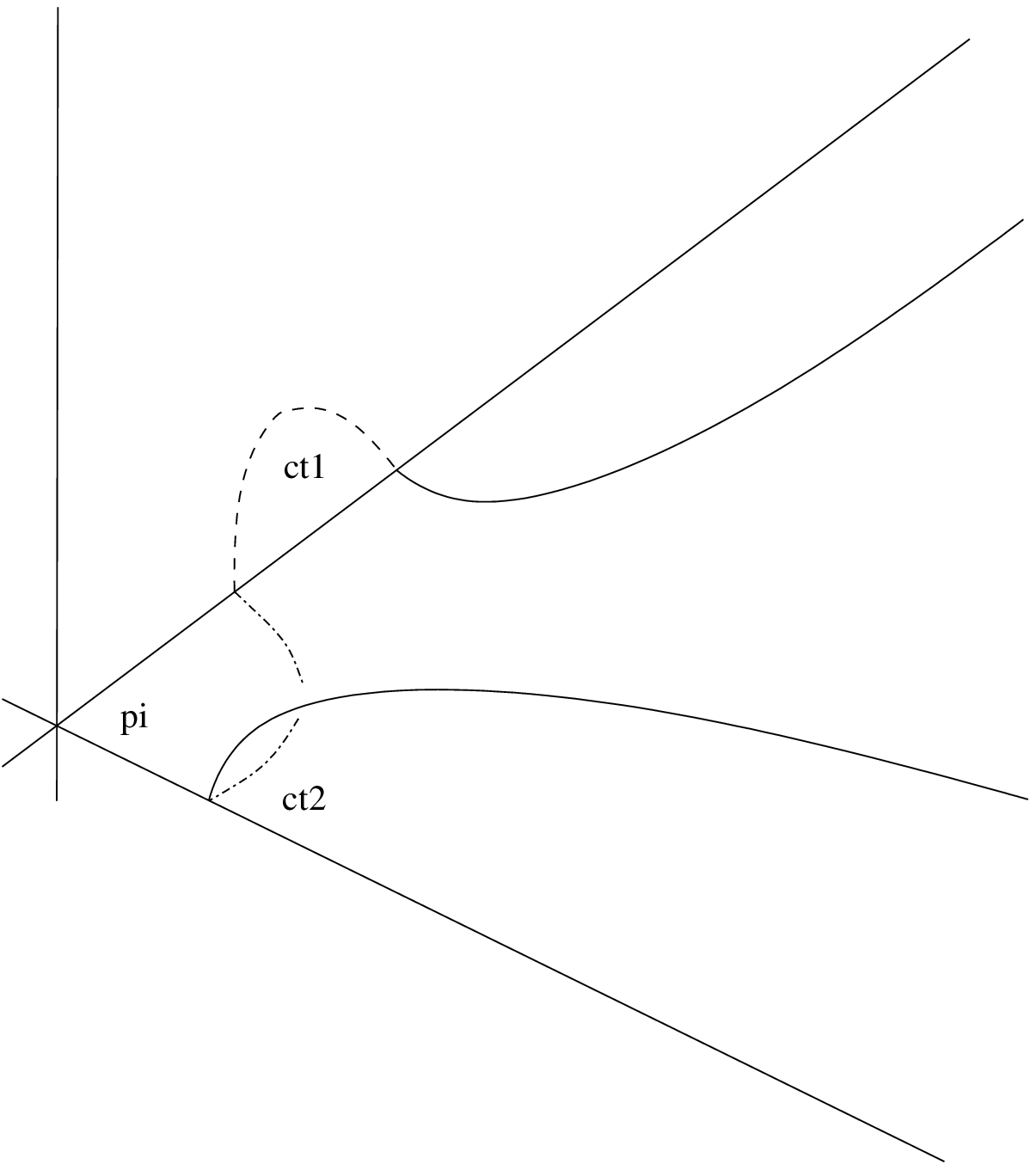}\hspace{1cm}
\includegraphics[width=0.4\textwidth]{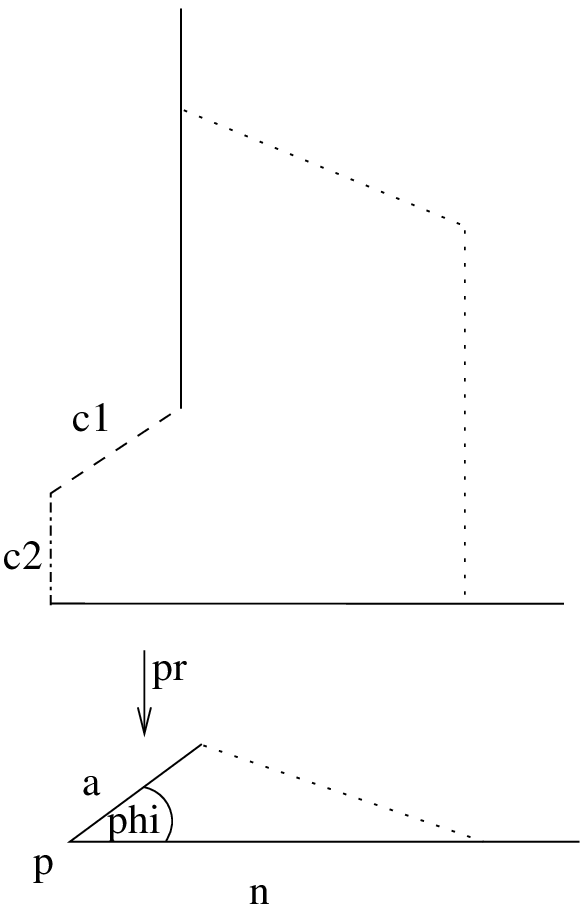}\end{center}
\caption{The desired boundary of the $1/2$-surface in $\H^2\times\R$ and its minimal sister surface in $\Nil$.}\label{f:genus1contour}
\end{figure}

The sister surface is bounded by a geodesic contour $\Gamma\coloneqq\Gamma(a,b,\phi)$: The horizontal mirror curves correspond to vertical geodesics and the vertical mirror curves correspond to horizontal geodesics, their projections enclose an angle $\phi\in\left(0,\pi\right)$.

The length $a>0$ of the finite horizontal geodesic $c_1$ determines the asymptotic parameter of the ends of the $k$-noid. The length $b>0$ of the vertical finite edge $c_2$ defines the size of the hole of the $k$-noid: For $b\to 0$ the $k$-noid is close to the non-degenerate $k$-noid from above, cf. Section~\ref{S:knoid}. The angle $\phi$ measures the curvature of $\tilde{c}_2$. The parameters will be determined by solving the period problems.

To construct a minimal surface that is bounded by $\Gamma$, we truncate the infinite contour $\Gamma$ and get closed Jordan curves $\Gamma_n$, $n>0$. We solve the Plateau problem for the closed Jordan curves and obtain a sequence of compact minimal surfaces. Afterwards we show there exists a minimal surface as a limit.

To define the closed Jordan curves $\Gamma_n$ we consider a geodesic triangle $\Delta_n$ in the base manifold $\R^2$ of the Riemannian fibration of $\Nil_3(\R)$:  Two edges of lengths $a$ and $n$ form an angle $\phi$ and intersect in a point $\hat{p}_1$. We lift $\hat{p}_1$ and the corresponding edge of length $n$ horizontally and label the vertices with $p_1$ and $p_6$. Then we add a vertical arc of length $b$ at $p_1$ in fibre-direction $\xi$ and label its endpoint with $p_2$. We lift the edge of length $a$ of the base triangle in $\R^2$ horizontally, such that it starts in $p_2$; the other vertex is labelled by $p_3$. We add another vertical edge in fibre-direction $\xi$: it starts in $p_3$, has length $n^2$ and its end vertex is called $p_4$. By lifting the remaining edge of the base triangle horizontally, such that it starts in $p_4$ and inserting another vertical edge with endpoints $p_5$ and $p_6$ we complete the special Jordan curve $\Gamma_n$.

The vertical distances do not sum up: $d(p_1,p_2)+d(p_3,p_4)\ne d(p_5,p_6)$, but we claim $p_6 p_5$ is in fibre direction. We consider an arc length parametrization $\gamma$ of $\partial\Delta_n\subset\R^2$ which runs counter-clockwise. By Lemma~\ref{l:vertdistances} the vertical distance of its horizontal lift is $\vol(\Delta_n)$. By construction we get
\[ d(p_6,p_5)=b+n^2-\vol(\Delta_n).\]
Since $\vol(\Delta_n)$ grows linearly, there exists $N\in\N$ such that for all $n\geq N$: $d(p_6,p_5)>0$ and $d(p_6,p_5)\to\infty$ for $n\to\infty$.

The polygon $\Gamma_n$ has six right angles; its projection $\Delta_n$ is convex for every $n$ and has one fixed angle $\phi\leq\pi$ independent of $n$. We define a mean convex set $\Omega_n\coloneqq\pi^{-1}(\Delta_n)\subset \Nil_3(\R)$.

For $n\to\infty$ we have $\Gamma_n\to\Gamma$ in the sense that $\Gamma_n\cap K_x=\Gamma\cap K_x$ for any compact neighborhood  $K\subset\Nil_3$ for $x\in\Gamma$ and $n$ large enough. 

\subsection{Plateau solutions}\label{SS:Plateau}
To control the Plateau solution for $\Gamma$, we solve the Plateau problem for $\Gamma_n$ first:

\begin{lemma}\label{l:plateaufinite}
The special Jordan curve $\Gamma_n\subset \Nil_3(\R)$ bounds a unique Plateau solution $M_n$. It is a section over $\Delta_n$ and extends smoothly without branch points by Schwarz reflection about the edges of $\Gamma_n$.
\end{lemma}
\begin{proof}
Since $\Delta_n$ is convex and $\partial\Delta_n$ consists of geodesics, its preimage $\pi^{-1}(\Delta_n)$ is a mean convex set since the preimage of each geodesic of $\partial\Delta_n$ is minimal. The intersection $\Omega_n$ of $\pi^{-1}(\Delta_n)$ with two horizontal halfspaces defined by the horizontal umbrellas in $p_1$ and $p_4$ (i.e. the exponential map of the horizontal spaces $(\T_{p_1}\Nil_3)^h$ and $(\T_{p_4}\Nil_3)^h$ respectively) as its boundaries is compact. Moreover, by construction we have $\Gamma_n\subset\partial\Omega_n$. Therefore, the solution of the Plateau problem exists and is embedded by~\cite{MY1982}. Moreover, by Section~\ref{S:Plateau} the solution $M_n$ is a unique section of the trivial line bundle $\pi\colon\Omega_n\to\Delta_n$. Proposition~\ref{P:boundarybranch} implies that it extends as a smooth immersion across $\Gamma_n$ by Schwarz reflection since at each vertex the angle is $\pi/2$.
\end{proof}

Actually we are interested in an infinite Plateau solution, we construct it as a limit. Define $\Delta\coloneqq\bigcup\limits_{n \in \R}\Delta_n$.
\begin{theorem}\label{t:plateauinfinite}
There exists a unique minimal surface $M(a,b,\phi) \subset \Nil_3(\R)$ with $\partial M=\Gamma$, which is a section on $\Delta$ and extends without branch points by Schwarz reflection about its edges for all $a,b>0$ and $\phi\in(0,\pi)$.
\end{theorem}
\begin{proof}
Consider the sequence of minimal sections $M_n=(x,u_n(x))$ defined on $\Delta_n$ by Lemma \ref{l:plateaufinite}. By the maximum principle $u_n$ is a monotone increasing sequence on $\Delta_k$, $n\geq k$. We claim, the sequence is uniformly bounded on each compact subset $K\subset\Delta$. Since by~\cite{RST2010} we obtain a gradient estimate in any $x\in\Delta'\subset K$ that depends on the distance of $x$ to the boundary and the upper bound, this implies the compactness.

To prove the claim we consider $k\in\N$ such that $K\subset\Delta_k$ and a horizontal helicoid $H_k\coloneqq H_\alpha$, where $\alpha$ depends on the parameters $(a,b,\phi)$ of the sequence $M_n$. We orient $H_k$ such that one of its vertical rulings coincides with the vertical geodesic of $\Gamma_n$ of length $n^2$. Moreover if we consider the Plateau solution $M_k$ on $\Delta_k$ and start with the helicoid in height $k^2$ such that $\pi(H_k)$ is bounded by $\pi(\overline{p_4p_5})$ from one side, then we can move the helicoid downwards up to a height $h_k$, since $M_k$ is a section and $H_k$ is still a barrier from above. Now consider the sequence of minimal sections $M_n$, by the maximum principle there is no point of contact. Hence, $M_n$ is uniformly bounded on each compact domain $K$.

By diagonalization we obtain a subsequence, call it $u_n$ again, converging to some minimal section $u$ on $\Delta$, the convergence is uniform on every compact subset of $\Delta$. The surface $M=(x,u(x))$ is a minimal surface of disc-type continuous up to the boundary $\Gamma$. By Proposition~\ref{P:boundarybranch} $M$ is an immersion that extends without branch points by Schwarz reflection. The surface is unique by Proposition~\ref{p:uniquesection}.
\end{proof}

For $b=0$ the proof still holds, which proves the existence of the minimal surface described in Section~\ref{S:knoid}.

\begin{remark}\label{R:monotone}
For $(a,\phi)$ fixed, the maximum principle implies that the sequence $(M(a,b,\phi))_b$ is monotone in $b$, in the sense that $M_1\coloneqq M(a,b_1,\phi)$ is barrier from above for $M_2\coloneqq M(a,b_2,\phi)$ for $b_1<b_2$. To see this we orient $M_1$ and $M_2$ such that their infinite vertical geodesics in the boundaries coincide in the end, their interior projects to disjoint domains in the base and their infinite horizontal geodesics in the boundaries are in heights $\epsilon>0$ and $-b_2$ resp. In the projection the two finite horizontal geodesics form an angle $\alpha=\pi-\phi>0$. If we increase this angle, i.e. we rotate one surface towards the other, it is clear that there is no inner point of contact. Furthermore the surfaces cannot intersect for any angle $\tilde{\alpha}>0$, since this would lead to two graphs with the same boundary. If we rotate the surface further such that they are graphs above the same domain, there is still no intersection since the normals rotate monotone along the vertical geodesic. The same argument shows that we can translate $M_1$ down without an intersection, which proves the claim.
\end{remark}

\subsection{Period problems}

To construct an mc $1/2$ surface with genus $1$ and certain symmetries we have to solve two period problems. One period is given by the vertical distance of the two horizontal mirror curves that are given by the two horizontal mirror curves. The second period is angular and determined by the geometry of the finite horizontal curve $\tilde{c}_2$.

To solve the first period problem for the mc $1/2$ surface $\tilde{M}\subset \H^2\times\R$, i.e. to construct a minimal surface $M\subset \Nil_3(\R)$ such that the two horizontal components of its sister surface lie in the same horizontal plane, we consider the mirror curve in the vertical plane with finite length $a$ in $\partial\tilde{M}$, and call it $\tilde{c}_1$. The period is given by $p=\int \la\tilde{c}'_1,\tilde{\xi}\ra_{\H^2\times\R}$, where $\tilde{\xi}$ is the vertical vector field of $\H^2\times\R$. Since we have a first order description for the vertical parts of vector fields, we consider the vector field $\tilde{T}$ on $\tilde{M}$ given by $\tilde{\xi}-\la\tilde{\xi},\tilde{\nu}\ra\tilde{\nu}$, where $\tilde{\nu}$ denotes the normal of $\tilde{M}$. It is the tangential projection of the vertical vector field and rotates by $\pi/2$ in the tangent plane under conjugation:
\[
\diff f^{-1}(T)=J\diff\tilde{f}^{-1}(\tilde{T}),
\]
where $f$ and $\tilde{f}$ denote the corresponding parametrizations of $M$ and $\tilde{M}$. Therefore, we have an analogous formulation of the period on $M$:
\begin{align*}
p(\tilde{M})&=\int\limits_{\tilde{c}_1} \la\tilde{c}'_1,\tilde{\xi}\ra_{\H^2\times\R}=\int \la\diff\tilde{f}(\gamma'),\tilde{T}\ra_{\H^2\times\R}\\
&=\int \la\diff\tilde{f}(\gamma'),\diff\tilde{f}(J^{-1}\diff f^{-1}(T))\ra_{\H^2\times\R}=\int \la J\gamma',\diff f^{-1}(T)\ra\\
&=\int\limits_{c_1}\la\eta,\xi\ra_{\Nil_3(\R)}=p(M).
\end{align*}

As seen above, we have a Plateau solution $M(a,b,\phi)$ for all $a,b>0$ and $\phi\in(0,\pi)$. The first period problem is solvable for each $a>0$ and $\phi\in\left(0,\pi/2\right)$:
\begin{proposition}\label{l:period1}
For each $a>0$ and $\phi\in\left(0,\pi/2\right)$ there exists $b(a,\phi)>0$ such that $p(M(a,b(a,\phi),\phi))=0$.
\end{proposition}
\begin{proof}
Let $a>0$ and $\phi\in\left(0,\pi/2\right)$ be fixed. By Theorem~\ref{t:plateauinfinite} we have a unique Plateau solution $M_b\coloneqq M(a,b,\phi)$ for each $b>0$. We claim the function $p(b)\coloneqq p(M_b)$ is continuous in $b$. To see this, we take two converging sequences $(b_l)$ and $(b_k)$ with the same limit $b_0$. If $\lim p(b_l)\ne \lim p(b_k)$ this would contradict the uniqueness of the minimal section, since the corresponding minimal surfaces also converge by the boundedness of the sequences $M_{l}\coloneqq M(a,b_l,\phi)$.

We will now show that there exists $b_t\in\R$ such that $p(b)<0$, for all $b>b_t$, and $\lim_{b\to 0}p(b)>0$. By the intermediate value theorem this proves the lemma.
\begin{itemize}
\item To define $b_t$, we consider the horizontal helicoid $M_H$ constructed by Daniel and Hauswirth~\cite{DH2009} from Section~\ref{S:conormalhoriheli}, whose horizontal axis coincides with $c_1$. By Lemma~\ref{l:a0alphainf} there exists a helicoid $M_H$ for each pitch $a>0$ such that the incident vertical arcs of $\Gamma$ are contained in its vertical rulings. Since $\phi<\pi/2$ the helicoid $M_H$ and $M_b$ are both minimal sections over a bounded convex domain $\pi(M_H)\cap\pi(M_b)$, whose boundary consists of three geodesic arcs. 

We claim there exists $b_t>0$, such that the surfaces intersect in $\Gamma$ only. In Section~\ref{S:conormalhoriheli} we showed that the conormal $\eta_H$ of the helicoid along the vertical ruling is horizontal and its opening angle depends continuously on the height given by 
\[h=\frac{\sinh\left(\frac{-1}{2}\ln\left(\frac{1+\cos\phi}{1-\cos\phi}\right)\right)}{2\alpha^2},\] where $\alpha$ depends on the pitch $a$, see Lemma~\ref{l:a0alphainf}. We consider the vertical plane $V$, given by its tangent plane, that is spanned by the conormal $\eta_H$ and $\xi$ in height $h$. Since each conormal is horizontal and turns monotonically, the intersection $V\cap M_H$ meets the vertical ruling in exactly one point given by the height $h$. Moreover $M_H$ is a section in the interior and therefore $V\cap M_H$ is bounded from below. Hence for $a>0$ and $\phi\in(0,\pi/2)$ exists $b_t\geq \abs{h}$ such that the helicoid is a barrier lying above $M_b$ for every $b> b_t$.

Consequently we can estimate the vertical component of the conormal $\eta$ of the Plateau solution $M_b$ by the helicoid conormal $\eta_H$ along the interior of the curve $c_1$: \[\la \eta,\xi\ra<\la\eta_H,\xi\ra\,\quad \mbox{ and hence }\quad p(M_b)=\int\la \eta,\xi\ra<\int\la\eta_H,\xi\ra=0.\]
\item Otherwise for $b\to 0$ we consider a minimal $k$-noid $N$ in $\Nil_3(\R)$ as in Section~\ref{S:knoid} with the parameters $a$ and $\phi$. It has a positive period $p_N$, since it is bounded from below by a horizontal umbrella.

For $b\to 0$ we have $M_b\to N$ away from the singularity. Furthermore, the sequence of conormals $\eta_b$ converges uniformly to $\eta_N$ on compact sets $K\subset c_1$. Therefore, the period $p(b)\vert_K$ converges uniformly to $p_N\vert_K >0$ for each compact $K\subset c_1$ and $b\to 0$. Hence, on $c$ we have $p(b)> 0$ for $b\to 0$.
\end{itemize}
\end{proof}

\begin{remark}
For $\phi=\pi/2$ the proof does not work since the helicoid $M_H$ is not a barrier for $b<\infty$. Therefore, we cannot construct a genus $1$ catenoid in $\H^2\times\R$ with this method.
\end{remark}

\begin{lemma}\label{l:bcont}
For each $a>0$ and $\phi\in\left(0,\pi/2\right)$ the map $b\colon \R_+\times\left(0,\pi/2\right)\to\R_+$ defined by Proposition~\ref{l:period1} is a continuous function.
\end{lemma}
\begin{proof}
We assume that $b$ is discontinuous in $(a,\phi)$, i.e. there exist two sequences $(a_l,\phi_l)$ and $(\overline{a}_l,\overline{\phi}_l)$ with limit $(a_0,\phi_0)$ but w.l.o.g. $b_0=\lim b(a_l,\phi_l)>\lim b(\overline{a}_l,\overline{\phi}_l)=\overline{b}_0$. In the proof of~\ref{l:period1} we have seen, that the sequences of the corresponding minimal surfaces contain converging subsequences with limits $M=M(a_0,b_0,\phi_0)$ and $\overline{M}=M(a_0,\overline{b}_0,\phi_0)$ respectively. Both minimal surfaces $M$ and $\overline{M}$ have zero period. But for $b_0>\overline{b}_0$ the minimal surface $M$ bounds a mean convex domain from above with $\partial \overline{M}$ in its boundary. Therefore, $M$ is an upper barrier for $\overline{M}$ with $\la\eta,\xi\ra>\la\overline{\eta},\xi\ra$, contradicting the fact that both minimal surfaces have zero period.\end{proof}

We solved the first period problem in $b$ depending on $(a,\phi)$. Namely, for each angle $\phi\in(0,\pi/2)$ and horizontal geodesic of length $a$, there exists $b>0$ as length of the vertical geodesic, such that the two horizontal mirror curves in the sister surface lie in the same mirror plane.

To solve the second period problem we need to restate the solution of the first period problem: For an angle $\phi$ and a vertical geodesic $c_2$ of length $b$ there exists a horizontal geodesic of length $a$, such that the period is zero:

\begin{proposition}\label{p:convergelimit}
For each $b>0$ and $\phi\in\left(0,\pi/2\right)$ there exists $a(b,\phi)>0$, such that the first period of the minimal surface $M(a(b,\phi),b,\phi)$ is zero.
\end{proposition}
\begin{proof}
By Lemma~\ref{l:bcont} the map $b\colon \R_+\times\left(0,\pi/2\right)\to\R_+$ continuous. For each $\phi$ we claim $b(a,\phi)\to 0$ for $a\to 0$. Indeed, by the proof of Proposition~\ref{l:period1} there exists $-h\in\R$ as an upper bound of $b(a,\phi)$ given by $a$ and $\phi$: \[-h=\frac{\sinh\left(\frac{1}{2}\ln\left(\frac{1+\cos\phi}{1-\cos\phi}\right)\right)}{2\alpha(a)^2}\geq b(a,\phi),\quad\text{since otherwise}\quad p>0.\] By Lemma~\ref{l:a0alphainf} we know that $\alpha(a)\to\infty$ for $a\to 0$. Therefore, the height $h$ converges to zero as well as $b(a,\phi)$ converges to zero.

Remains to show that the map $b$ is unbounded. For each $\phi>0$ assume the contrary: The function $b_\phi(a)\coloneqq b(a,\phi)\leq \hat{b}$ is bounded. Since $b$ is a continuous function, we have $p(M_{\tilde{b}})\ne0$ for $\tilde{b}>\hat{b}$. By Remark~\ref{R:monotone} the minimal surface $M(a,b(a,\phi),\phi)$ is a barrier from above for $M(a,\tilde{b},\phi)$, therefore $p(M_{\tilde{b}})<0$. The continuity of $p$ in $b$ implies $p(M_{\tilde{b}})<0$ for $\tilde{b}>\hat{b}$ and for all $a>0$.

Consider the map $p_1\colon a\mapsto p(M(a,\tilde{b},\phi))$. We claim the function $p_1$ is continuous. To see this, we take two converging sequences $(a_l)$ and $(a_k)$ with the same limit $a_0$. Since the corresponding minimal surfaces also converge, this implies $\lim p(a_l)= \lim p(a_k)$. Moreover, for $a$ large enough the minimal surface has a positive period, therefore there exists $\hat{a}>0$ with $p_1(\hat{a})=0$, which is a contradiction.

Hence, $b_\phi(a)$ is unbounded with $b_\phi(a)\to0$ for $a\to0$. So we conclude: For all $b>0$ there exists $a(b,\phi)>0$ (not necessarily unique) such that $M(a(b,\phi),b,\phi)$ has zero period.
\end{proof}

\begin{remark}
For conjugate minimal surfaces in $\R^3$ we have a unique $a(b,\phi)$, since $a\mapsto b(a,\phi)$ is injective because of scaling.
\end{remark}

The second period problem relies on the horizontal mirror curve $\tilde{c}_2$ in the cmc sister $\tilde{M}$ and there are two difficulties, we have to solve: First of all, we have to ensure that the two vertical mirror planes, which are perpendicular to $\tilde{c}_2$ intersect. And secondly their angle of intersection has to be $\pi/k$. We have to choose the pair $(b,\phi)$ of the minimal surface such that the sister surface fulfils the desired properties. The second period problem is independent of $a$ since it is given by its curvature and length, therefore we do not need a degree argument like in~\cite{KPS1988}. Hence it makes sense to solve the second period problem by considering $\tilde{c}_2$ only.

Before we solve the second period problem, we want to analyse the finite horizontal mirror curve $\tilde{c}_2$ in $\H^2\times\R$ parametrized by arc length. We choose the downward-pointing surface normal $\nu$ and $(c'_2,\eta,\nu)$ positive oriented. Where $c_2$ is the sister curve, then $\la c'_2,\xi\ra=1$. We consider the twist $\phi(t)$ of $c_2$, recall the definition from Section~\ref{S:sistercurves}. The curvature of $\tilde{c}_2$ is $\tilde{k}=1-\phi'(t)<1$, since $\phi'>0$ by the graph property of the minimal surface. Proposition~\ref{p:anglecompare} implies the embeddedness, since $\theta\leq\phi<\pi/2$; moreover, $\theta\geq0$ since $\theta<0$ implies $\tilde{k}>1$. With $\gamma_0$ and $\gamma_b$ we denote the unique geodesics given by $\gamma_i(0)=\tilde{c}_2(i)$ and $\gamma'_i(0)=-\tilde{\nu}(i)$, $i=0,b$, see Figure~\ref{f:intersect}.
\begin{figure}
\begin{center}
\psfrag{c}{$\tilde{c}_2$}
\psfrag{0}{$\gamma_b$}
\psfrag{1}{$\gamma_0$}
\psfrag{a}{$\alpha$}
\includegraphics[width=4.5cm]{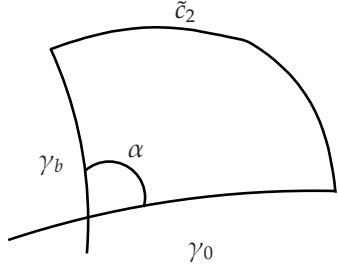}\end{center}
\caption{Sketch of the defined geodesics $\gamma_i$, $i=0,b$.}\label{f:intersect}
\end{figure}

As said before, in a first step we have to ensure, that the geodesics intersect:
\begin{proposition}\label{p:intersect}
For each $\phi\in\left(0,\pi/2\right)$ exists $b_0\coloneqq b_0(\phi)\in(0,\phi)$, such that the vertical mirror planes of $\tilde{M}(a(b,\phi),b,\phi)$ intersect for all $b\in(0,b_0(\phi))$ and define an intersection-angle $\alpha>0$.
\end{proposition}
\begin{proof}
As in Proposition~\ref{p:anglecompare} we consider the foliation by horocycles given by $\tilde{\nu}(0)$ and the related angle $\theta\leq\phi<\pi/2$. For the calculation we consider the upper halfplane and orient $\tilde{c}_2$ such that $\tilde{c}_2(0)=(0,1)$ and $\tilde{\nu}(0)=(0,1)$, then $\gamma_0$ is contained in the $y$-axis. We want to parametrize the unique geodesic $\gamma_b\subset\H^2$, which starts in the endpoint of $\tilde{c}_2$ ($\tilde{c}_2(b)=(c_x,c_y)$) and its tangent is parallel to $(\sin(\theta), -\cos(\theta))$. For $\theta\ne k\pi, k\in \Z$, $\gamma_b$ is an Euclidean halfcircle with radius $r$ and its midpoint on the $x$-axis. We solve the linear equations
\[
\gamma_b(\pi-\theta)=(c_x,c_y)=(x+r\cos(\pi-\theta),r\sin(\pi-\theta)).\] The geodesic is in Euclidean coordinates parametrized by \[\gamma_b(t)=\left(c_x+c_y \left(\frac{\cos(\pi-t) -\cos \theta}{\sin\theta}\right),c_y\frac{\sin(\pi- t)}{\sin\theta}\right).\]
The geodesics intersect if the $x$-coordinate of $\gamma_b(\pi)$ is positive: \begin{equation}\label{E:inequal}
c_x+c_y\frac{1-\cos\theta}{\sin\theta}>0.
\end{equation} 
Since $d_{\H^2}(\tilde{c}_2(0),\tilde{c}_2(b))\leq b$ we have $c_x\geq -b$ and $c_y>e^{-b}$. So we conclude Equation~\eqref{E:inequal} is true if \[\frac{1-\cos\theta}{\sin\theta}>b e^b.\]
The angles $\theta$ and $\phi$ are related by $\theta'=\phi'+\cos\theta-1$ (see Proposition~\ref{p:anglecompare}) and $\cos\theta\geq0$ implies $\int(\cos\theta-1)\geq-b$. Therefore we know that $\theta\geq\phi-b$. Furthermore, the function $\theta\mapsto (1-\cos\theta)/\sin\theta$ increases monotonically, so we conclude, the geodesics intersect if
\[be^b<\frac{1-\cos(\phi-b)}{\sin(\phi-b)}.\]
This is equivalent to \begin{equation}\label{E:implicitb}
f_\phi(b)\coloneqq \frac{1-\cos(\phi-b)}{\sin(\phi-b)}-b e^{b}>0.\end{equation} Its differential \[f'_\phi(b)=\frac{\cos(\phi-b)-1}{\sin^2(\phi-b)}-e^b(1+b)\] is less than zero for all $b<\phi$. Hence, $f_\phi$ is decreasing. Let us consider the limits at the boundaries:
\begin{align*}
\lim_{b\to0}f_\phi(b)&=\frac{1-\cos\phi}{\sin\phi}>0,\qquad \text{for }\phi\in(0,\pi/2)\\
\lim_{b\to\phi}f_\phi(b)&=-\phi e^\phi<0.
\end{align*}
We conclude, there exists exactly one $b_0(\phi)\in(0,\phi)$ with $f_\phi(b_0(\phi))=0$. Moreover, for all $b<b_0(\phi)$ we have $f_\phi(b)>0$, therefore the geodesics intersect for all $b<b_0(\phi)$.

The vertical mirror planes are given by $\gamma_i\times\R\subset\H^2\times\R$ for $i=0,b$.
\end{proof}
\begin{remark}
For $b=\phi$ the total curvature of $\tilde{c}_2$ is $b-\phi=0$ and therefore the two geodesics $\gamma_0$ and $\gamma_b$ do not intersect in any $p\in\H^2$, but in $\partial\H^2$. By Proposition~\ref{p:convergelimit} there exists $a>0$ to solve the first period problem, hence after reflection the contruction causes a complete singly periodic mc-$1/2$ surface $\tilde{M}(a(\phi,\phi),\phi,\phi)$ in $\H^2\times\R$ with infinitely many ends.
\end{remark}

Now we are able to solve the second period problem that is given by the angle $\alpha$. We want the surface to close after $2k$ reflections about vertical mirror planes, so $\alpha$ has to be $\pi/k$.

\begin{proposition}\label{P:2periods}
For each $k\geq3$ there exists $\epsilon=\epsilon(k)>0$, such that $\phi=\pi/k+\epsilon<\pi/2$ and there exists $0<b<b_0(\phi)$, such that the surface $\tilde{M}(a,b,\phi)$ has angular period $\alpha=\pi/k$.
\end{proposition}
\begin{proof}
The angular period is given by the intersection angle of the vertical mirror planes, i.e. the intersection angle of the two geodesics. For $b<b_{0}(\phi)$ the geodesics $\gamma_0$ and $\gamma_b$ intersect by Proposition~\ref{p:intersect}. Hence, we can apply the Gau\ss{}-Bonnet Theorem to the compact disc $V\subset\H^2$ defined by $\partial V=\tilde{c}_2\cup\gamma_0\cup\gamma_b$:
\[\int\limits_V K+\int\limits_{\partial V} k_g+\sum \alpha_i=2\pi\chi(V),\]
where $\alpha_i,\,i=1,2,3$ are the exterior angles with $\alpha_i=\pi/2$ for $i=1,2$ and $\alpha_3=\pi-\alpha$. Furthermore, $k_g$ is the geodesic curvature with respect to the inner normal of $\partial V$ and since $\tilde{c}_2$ is a mirror curve $k_g=-\tilde{k}$ with respect to the surface normal.
Using this we get
\begin{equation}\label{E:gaussbonnet}\int\limits_V K-\int\limits_{\tilde{c}_2} \tilde{k}+2\pi-\alpha=2\pi.\end{equation}
From Lemma~\ref{l:torsionangle} we know $\tilde{k}=1-\phi'$. Integrating this shows Equation~\eqref{E:gaussbonnet} is equivalent to $\phi-b-\vol(V)=\alpha$. 

We claim, that the lengths $l(\gamma_i)$ are bounded from above for all $b\leq b_{0}(\phi)$. Recall from the proof of Proposition~\ref{p:intersect} that $\phi\geq\theta\geq\phi-b$. In particular, we have a lower bound for $\theta$, the angle that is given by the tangent of $\tilde{c}_2$ and the horocycle fibration. Assume that the lengths $l(\gamma_i)(b)\to\infty$ for $b\to 0$, this implies $\theta\to0$, a contradiction. If the lengths are bounded, the area tends to zero for $b\to 0$.

The angle $\alpha$ depends continuously on $b<b_{0}(\phi)$: \begin{equation}\label{E:alpha}\alpha(b)=\phi-b-\vol(V(b))\end{equation} and decreases. The idea is to show that for any $k\geq3$ exists $\phi_k\in(0,\pi/2)$ such that 
\[\lim_{b\to 0}\alpha(b)>\pi/k\quad \text{and}\quad \lim_{b\to b_0(\phi_k)}\alpha(b)<\pi/k.\]

On the one hand $\lim_{b\to 0}\alpha(b) =\phi_k$, hence we have to choose $\phi_k>\pi/k$. Therefore, for any $\epsilon_k>0$, $\phi_k\coloneqq\pi/k+\epsilon_k$ satisfies the first condition.

On the other hand we have \[ \lim_{b\to b_0(\phi_k)}\alpha(b)=\phi_k-b_0(\phi_k)-\vol(V(b_0(\phi_k))),\] hence we have to choose $\epsilon_k>0$ such that \[\phi_k<\frac\pi k+b_0(\phi_k)+\vol(V(b_0(\phi_k)))\quad\Leftrightarrow\quad \epsilon_k<b_0(\phi_k)+\vol(V(b_0(\phi_k))).\] We claim the function $b_0(\phi)$ increases monotonically. Recall Equation~\eqref{E:implicitb}: $b_0$ was defined implicitly by $f_\phi(b)= \frac{1-\cos(\phi-b)}{\sin(\phi-b)}-b e^{b}=0$. By the chain rule we have \[b'_0(\phi)=e^{b_0(\phi)}(1+b_0(\phi))+\frac{1-\cos(\phi-b_0(\phi))}{\sin^2(\phi-b_0(\phi))}>0.\] Therefore with $\epsilon_k=b_0(\pi/k)$ we get:
\[\lim_{b\to b_0(\phi_k)}\alpha(b)=\pi/k+\underbrace{b_0(\pi/k)-b_0(\phi_k)}_{<0}-\vol(V(b_0(\phi_k)))<\pi/k.\] Remains to show that $\phi_k<\pi/2$ for all $k\geq3$. Since $b_0$ increases this is true if $\pi/3+b_0(\pi/3)<\pi/2$, i.e. $b_0(\pi/3)<\pi/6$. But this follows directly from the fact that $f_{\pi/3}(\pi/6)<0$.

By the intermediate value theorem follows, there exists $b^*\in(0,b_{0}(\pi/k+\epsilon_k))$ such that $\alpha(b^*)=\pi/k$.
\end{proof}

The proposition proves the existence of one pair $(b_k,\phi_k)$ for each $k\geq3$ such that the angular period is $\pi/k$. It is natural to ask if there is a family of cmc surfaces with this angular period and $k$ ends. The answer is yes:

\begin{proposition}\label{P:neighb}
For each $k\geq3$ there exists an interval $U_k\subset(0,\pi/2)$, such that for all $\phi\in U_k$ there exists $b(\phi)>0$, such that each cmc surface $\tilde{M}(a,b(\phi),\phi)$ has angular period $\alpha=\pi/k$.
\end{proposition}
\begin{proof}
In Proposition~\ref{P:2periods} the existence of a pair $(b(\phi_k),\phi_k)$ was proven such that the surface $\tilde{M}(a,b(\phi_k),\phi_k)$ has the desired property. By Equation~\eqref{E:alpha} we know, the pair $(b(\phi_k),\phi_k)$ is a zero of the following continuously differentiable function
\[
G(b,\phi)=\alpha(b,\phi)-\pi/k=\phi-b-\vol(V(b,\phi))-\pi/k.\] 
The angle $\alpha(b,\phi)$ is given by the two geodesics $\gamma_0$ and $\gamma_b$. Recall from the proof of Proposition~\ref{P:2periods} that $\partial_b\alpha(b,\phi)<0$. By the implicit value theorem there exists an open neighborhood $U_0$ of $\phi_k$, an open neighborhood $V$ of $b(\phi_k)$, and a unique continuously differentiable function $g\colon U_0\to V$ with $g(\phi_k)=b(\phi_k)$ such that $G(g(\phi),\phi)=0$ for all $\phi\in U_0$.

To define $U_k$ notice that $g(\phi_k)<b_0(\phi_k)$, therefore the subset $U_k\coloneqq \{\phi\in U_0\colon g(\phi)<b_0(\phi)\}\cap(0,\pi/2)$ is not empty. Hence, for all $\phi\in U_k$ there exists $b=g(\phi)<b_0(\phi)$ such that $\alpha=\pi/k$.
\end{proof}

\begin{remark}
We want to analyse the limiting cases:
\begin{itemize}
\item $\phi\to\inf U_k\geq\pi/k$: From $\phi\to\pi/k$ follows \mbox{$b+\vol(V(b,\phi))\to 0$}, which implies $b\to0$. Assuming the solution of the first period problem $a(b,\phi)>0$ for $b\to0$ leads to the $k$-noid from Section~\ref{S:knoid} which has a positive period. Therefore if $\inf U_k=\pi/k$ then the sequence of $k$-noids converges to an union of $k$ horocylinders away from the singularity for $\phi\to \pi/k$.
\item $\phi\to\sup U_k\leq \pi/2$: Since $\partial_\phi\vol(V(b,\phi))\leq0$,  $\phi\to\sup U_k$ implies that $b$ increases. For the cmc surfaces this means that the length of the finite horizontal symmetry curve grows.
\end{itemize}
\end{remark}

\subsection{Main Theorem}

After solving the two period problems we can now prove the existence of the mc $1/2$ surface with genus $1$:
\begin{theorem}
For $k\geq 3$, there exists a family of surfaces $M$ with constant mean curvature $1/2$ in $\H^2\times\R$ such that:

\noindent
$\bullet\, M$ is a proper immersion of a torus minus $k$ points,

\noindent
$\bullet\, M$ is Alexandrov embedded.

\noindent
$\bullet\, M$ has $k$ vertical mirror planes enclosing an $\pi/k$-angle,

\noindent
$\bullet\, M$ has one horizontal mirror plane.
\end{theorem}
\begin{proof}
For $k\geq3$ we consider $(b(\phi),\phi)$ for $\phi\in U_k\subset(0,\pi/2)$ given by~\ref{P:neighb}. The minimal surface $M_\phi=M(a(b(\phi),\phi),b(\phi),\phi)$ defined by~\ref{p:convergelimit} solves the first period problem. By~\cite{daniel2007} the fundamental piece $M_\phi$ has a sister surface $\tilde{M}$ with constant mean curvature $1/2$ in $\H^2\times\R$, which is a graph. By construction and from the solution of the period problems, $\tilde{M}$ has one horizontal and two vertical mirror planes; the two vertical mirror planes enclose an angle $\pi/k$. It consists of four mirror curves: two horizontal (one bounded and one unbounded) and two vertical (one bounded and one unbounded). After Schwarz reflection about one of the vertical mirror planes followed by reflection about the horizontal mirror plane we have completed one end: It is built up of four fundamental pieces $\tilde{M}$. We use the Euler characteristic 
\[
\chi=V-E+F=2-2g
\]
to determine the genus $g$ of the complete mc $1/2$ surface $M$ with $k$ ends, which is generated by Schwarz reflection. We have $\chi=4k-8k+4k=0$ and therefore $g=1$.

\begin{figure}
\begin{center}
\psfrag{c1}{$c_1$}
\psfrag{c2}{$c_2$}
\psfrag{a}{$a$}
\psfrag{n}{$n$}
\psfrag{pr}{$\pi$}
\psfrag{phi}{$\phi$}
\psfrag{pi}{$\pi/k$}
\psfrag{ct1}{$\tilde{c}_1$}
\psfrag{ct2}{$\tilde{c}_2$}
\psfrag{ct0}{$\tilde{c}_0$}
\psfrag{ct3}{$\tilde{c}_3$}
\psfrag{H}{$\H^2\times\{0\}$}
\includegraphics[width=0.7\textwidth]{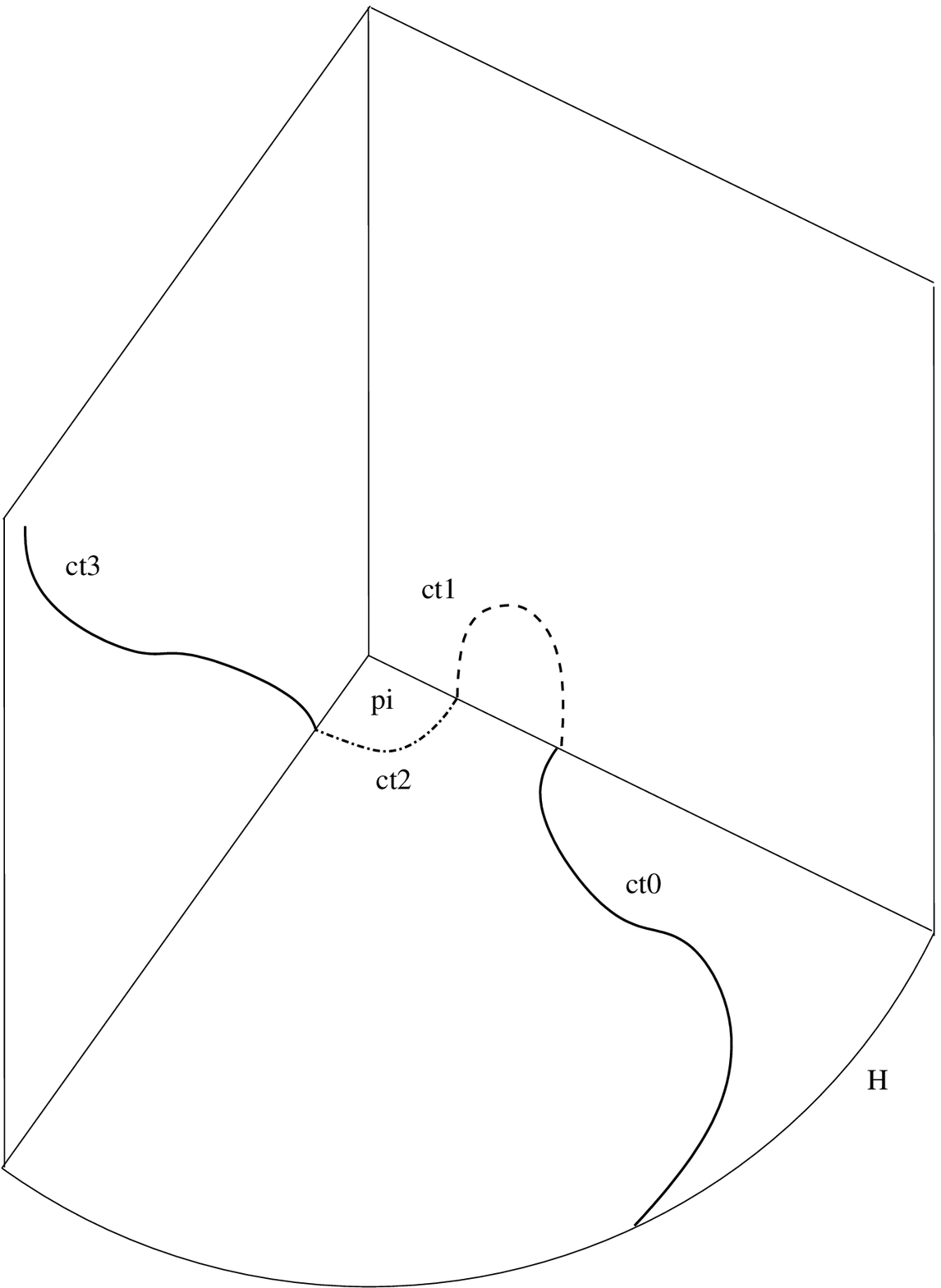}
\end{center}
\end{figure}

We claim that $M$ is Alexandrov-embedded if we choose as before the downward-pointing normal $\nu$. We show the fundamental piece $\widetilde{M}$ is embedded and stays in the subset of $\H^2\times\R$ that is bounded by mirror planes. The mirror planes are given by the symmetry curves with infinte length. We discuss the embeddedness of each boundary arc, since $\widetilde{M}$ is transverse to the fibres, the fundamental piece is embedded. Therefore, the complete surface $M$ is Alexandrov-embedded.

Let us recall the notation, $c_1$ denotes the finite horizontal geodesic in $\partial M_\phi$ and $\tilde{c}_1$ its sister curve in $\partial\widetilde{M}$. 
The sister curve $\tilde{c}_1$ is a mirror curve in a vertical plane. The curve is graph and therefore embedded. The same holds for the other horizontal curve and its sister curve in a vertical plane, call it $\tilde{c}_3$.

Along the horizontal geodesic $c_3$ we have $\la\eta,\xi\ra>0$ if we choose $(c_3'\eta,\nu)$ positive oriented. By the first order description of the sister surfaces we know that the projection of the vertical vector field $\xi$ on the tangent plane rotates by $\pi/2$ under conjugation. Therefore, we have $\la\tilde{c}_3',\tilde{\xi}\ra<0$. By construction, we know that $\la\eta,\xi\ra\to 1$ in the end along the sister curve $c_3$. Hence, in the sister surface for the corresponding mirror curve we have $\la\tilde{c}'_3,\tilde{\xi}\ra\to -1$. In other words, the mirror curve comes from $\infty$ in the end.

With the arguments of Hauswirth, Rosenberg and Spruck in the proof of Theorem 1.2 in~\cite{HRS2008} we know that each divergent sequence $(p_n)$ in $M$ with \linebreak \mbox{$\la\tilde{\nu}(p_n),\tilde{\xi}\ra\to0$} has a limit in the boundary in the projection: $\pi(p_n)\to\partial\H^2$. The idea of the proof is to show that if the sequence has a horizontal normal in the limit and a limiting point in the projection, then the whole surface is converging to a horocylinder and stays on one side. Hence, by the half-space theorem it is a horocylinder which is a contradiction. 

As in the setup let $c_2$ denote the finite vertical geodesic, recall that its sister curve is embedded. Let $c_0$ denote the remaining vertical geodesic in the boundary $\partial M_\phi$ parametrized in $\xi$-direction. Let $\alpha$ denote the twist of the horizontal normal $\nu$ along $c_0$. We have $\alpha'>0$, i.e. the normal rotates monotonically, because $M_\infty$ is a section. We consider the horocycle fibration of $\H^2$ given by $\tilde{\nu}(0)$ as in Proposition~\ref{p:anglecompare}. By the proposition the angle $\theta$ the sister curve encloses with the fibration in $\H^2$ is smaller than $\alpha=\pi-\phi<\pi$. Hence, $\tilde{c}_0$ is embedded. In summary this proves the first part, $\widetilde{M}$ is embedded.

In the second step we show that $\widetilde{M}$ is bounded by symmetry planes. First, we have to check if the surface stays in a horizontal halfspace of $\H^2\times\R$. We define the horizontal mirror plane of $M$ to be $\H^2\times\{0\}$. By the maximum principle $\widetilde{M}$ has no (global) minimum below $\H^2\times\{0\}$.

Assume there is a curve $\tilde{c}$ in $\widetilde{M}$ with $\tilde{c}(0)\in\partial \widetilde{M}$, whose third component goes to minus infinity. Then $\la\tilde{c}'(t),\tilde{\xi}\ra<0$ for $t>T$ for some $T\in\R$. This implies for the conormal $\la\eta(t),\xi\ra>0$ along the sister curve $c\subset M_\phi$ in the end. Let us recall that $M_\phi$ was constructed as a limit of compact minimal sections $M_n$. Hence there exists a $N\in\N$ such that $\T_p M_N$ is horizontal. The intersection $V$ of $M_N$ and the horizontal umbrella in $p$ consists of $2m$ curves, $m\geq2$. But this implies there exists a loop in $V$ which contradicts the uniqueness of the minimal section. Therefore $\widetilde{M}\subset \H^2\times\R^+_0$.

To finish the proof we have to show that the surface lies on one side of the vertical symmetry plane: We have seen that the projection of the vertical symmetry curve $\tilde{c}_3$ is converging to some point in the boundary $\partial\H^2$. Along $\tilde{c}_0$ the surface is at least locally to the side of the normal, since $\tilde{k}=1-\alpha'<1$ by the graph property of the minimal surface along $c_0$. Therefore, since it is also a graph, it is defined on a domain $\Omega\subset\H^2$ given by $\pi(\tilde{c}_3)\cup\tilde{c}_0$.
\end{proof}

\subsection{Conclusion and outlook}

We constructed an mc $1/2$ surface in $\H^2\times\R$ with $k$ ends, genus $1$ and $k$-fold dihedral symmetry, $k\geq3$, which is Alexandrov embedded. We had to solve two period problems in the construction. The first period guarantees that the surface has exactly one horizontal symmetry. For the second period we had to control a horizontal mirror curve to get the dihedral symmetry. In the case of $H\ne0$ the total curvature of a horizontal mirror curve depends not only on the twist of the normal, but also on its length. An interesting problem is to construct non-embedded examples that correspond to those presented here. Gro\ss{}e-Brauckmann proved in~\cite{brauckmann1993} that each Delaunay surface is the associated mc $1$ surface of a helicoid in $\S^3$. The family of Delaunay surfaces consists of embedded (unduloid) and non-embedded (nodoid) examples. In order to construct non-Alexandrov embedded examples one has to consider the minimal surface from Proposition~\ref{p:convergelimit}, and choose the positive oriented surface, i.e. the upward-pointing surface normal $\nu$. If $(c'_i,\eta,\nu),\,i=0,2$ is positively oriented, then $\la c'_i,\xi\ra=-1$. Therefore the twist $\alpha$ is decreasing which implies $\tilde{\kappa}=1-\alpha'>1$ for $\tilde{c}_{0/2}$. With this setup one has to solve the second period problem.

As far as the author knows it is an open problem to show the convergence of the ends of the cmc surface. It would be nice to prove, that $M$ has catenoidal ends, in the sense they are described in~\cite{DH2009}. To show this one may consider the minimal surface in $\Nil_3(\R)$ and wrap it between two copies of horizontal helicoids from~\cite{DH2009} which differ by a vertical translation about height $h$:
\begin{align*}
H^1_1(u_1,v_1)&=\frac{\sinh(\alpha v_1)}{\alpha(\psi'(u_1)-\alpha)}\cos\psi(u_1)\\
H^2_1(u_1,v_1)&=-G(u_1)\\
H^3_1(u_1,v_1)&=\frac{-\sinh(\alpha v_1)}{\alpha(\psi'(u_1)-\alpha)}\sin\psi(u_1), \text{ and}\\
H^1_2(u_2,v_2)&=\frac{\sinh(\alpha v_2)}{\alpha(\psi'(u_2)-\alpha)}\cos\psi(u_2)\\
H^2_2(u_2,v_2)&=-G(u_2)\\
H^3_2(u_2,v_2)&=\frac{-\sinh(\alpha v_2)}{\alpha(\psi'(u_2)-\alpha)}\sin\psi(u_2)+h.
\end{align*}
Where $\alpha$ is chosen such that the width of the helicoid is $a\sin\phi$, which is possible by Lemma~\ref{l:a0alphainf}.

The idea is to prove exponential convergence of the two helicoids, i.e. there exists $C$, $\lambda\in\R$ independent of $x\in H_1$ such that 
\begin{equation}\label{E:expconv}
d(x,H_2)<C e^{-\lambda\abs{x}},\quad \text{for }\abs{x}\to\infty.
\end{equation}
Since the Plateau solution $M(a,b_0)$ is bounded by $H_1$ and $H_2$ for an appropiate choice of $h$, this proves that $M_\phi$ has a helicoidal end and since the convergence is exponential this translates to the sister surface. 

\bibliographystyle{amsalpha}
\bibliography{/Users/juliaddg/Documents/bibliography}

\end{document}